\documentclass[10pt]{article}

\usepackage[margin=1.0in]{geometry}
\usepackage{amsmath,amssymb,amsthm,mathtools}
\usepackage{xcolor}
\usepackage{enumitem}
\usepackage[colorlinks=true,linkcolor=blue!55!black,citecolor=blue!55!black]{hyperref}
\usepackage{comment}
\usepackage{longtable}
\usepackage{array}
\usepackage{tikz}
\usetikzlibrary{arrows.meta,positioning,fit,backgrounds}

\usepackage{setspace}

\theoremstyle{plain}                 % uncomment this — it was commented out
\newtheorem{theorem}{Theorem}[section]
\newtheorem{lemma}[theorem]{Lemma}

\theoremstyle{definition}

\theoremstyle{remark}

\newtheoremstyle{mainthm}
  {1.2\topsep}   % space above
  {1.2\topsep}   % space below
  {\itshape}     % body font
  {}             % indent
  {\bfseries}    % head font
  {.}            % punctuation after head
  {\newline}     % space after head — \newline puts the body on its own line
  {}             % head spec (empty = default)
\theoremstyle{mainthm}
\newtheorem{theoremABC}{Theorem}

\theoremstyle{plain}                 % restore for anything declared later

\newcommand{\HH}{\mathbb{H}}
\newcommand{\PP}{\mathbb{P}}
\newcommand{\CC}{\mathbb{C}}

\newcommand{\ZZ}{\mathbb{Z}}
\newcommand{\Res}{\operatorname{Res}}
\newcommand{\ord}{\operatorname{ord}}

\newcommand{\Gz}{\Gamma}
\newcommand{\hgz}{\mathbf{H}_\Gz}
\newcommand{\Xg}{X_{\Gamma}}

\newcommand{\PSL}{\mathrm{PSL}}
\newcommand{\Stab}{\mathrm{Stab}}
\newcommand{\vol}{\mathrm{vol}}

\definecolor{newred}{RGB}{190,25,25}

\title{\bfseries On a generating function of the basis of weakly holomorphic functions on an elliptic curve}
\author{Joshua S. Friedman\footnote{The views expressed in this article are the author's own and not those of the U.S. Merchant Marine Academy, the Maritime Administration, the Department of Transportation, or the United States government.}, Jay Jorgenson\thanks{The second named author acknowledges grant support from PSC-CUNY Awards 67415-00\,55 and 68462-00\,56, which are jointly funded by the Professional Staff Congress and The City University of New York.} and Lejla Smajlovi\'c\thanks{{Generative AI Statement: a large language model (Anthropic) was used as an assistive software tool for nuerical investigation, Python/SageMath and Pari/GP script generation, and logical error checking. All main mathematical results were obtained by the authors. All scripts used to perform symbolic calculations for this paper were independently verified and proven correct by the authors, who assume full responsibility for all manuscript content.}}}
\date{\today}

\begin{document}
\maketitle

%\doublespacing

\begin{abstract}
Let $\Gz$ be a cofinite Fuchsian group whose quotient $\Gz\backslash\HH$ has genus one and a
single cusp, normalized to be at $\infty$ with width one. Let $\Xg$ be its smooth
compactification, which is an elliptic curve over $\CC$. Let $x,y$ be canonical generators of
the function field on $\Xg$, which have poles of order $2,3$ respectively at the cusp and which
{satisfy} a generalized Weierstrass relation $y^2+(a_1x+a_3)y=x^3+a_2x^2+a_4x+a_6$. Let $f$ be
the unique (up to scale) weight two cusp form for $\Gz$, normalized so that the associated
holomorphic differential is $\omega=dx/(2y+a_1x+a_3)$. With all this, we prove that the generating
function identity
\[
\frac{\bigl(y(\tau)+y(z)+a_1x(z)+a_3\bigr)f(z)}{x(z)-x(\tau)}=\sum_{m\ge0}\Phi_m(\tau)q_z^m
\]
defines a family $\{\Phi_m\}_{m\ge0}$ of weakly holomorphic modular functions on $\Xg$ such that
the set $\{\Phi_0\}\cup\{\Phi_m\}_{m\ge2}$ is a canonical basis of the space
$M_{0,\Gz}^{!,\infty}$ of weakly holomorphic modular functions on $\Xg$ with poles supported only
at the cusp. As an application, by comparing the generating function of the family
$\{\Phi_m\}_{m\ge0}$ with the generating function of the Niebur--Poincar\'e series, we show that
the generating function of the special values of the Kloosterman zeta function at $1$ is the
holomorphic part of a certain weighttwo harmonic Maass form up to an additive constant.
\end{abstract}

\section{Introduction}\label{sec:intro}

For the full projective modular group $\PSL_2(\ZZ)$, the elliptic
modular invariant
\[
j(\tau)=q_\tau^{-1}+744+196884q_\tau+\cdots,
\qquad q_\tau=e^{2\pi i\tau},
\]
generates the function field of the modular curve $X(1)=\PSL_2(\ZZ)\backslash\HH$. This is a genus-zero Riemann surface with a cusp at $\infty$ of width one. Set $j_1(\tau):=j(\tau)-744$ and, for $m\geq 1$, let
$J_m(\tau):=j_1|T_m(\tau),$
where $T_m$ denotes the $m$-th Hecke operator. Together with
$J_0\equiv 1$, the functions $\{J_m\}_{m\geq 0}$ form a canonical
family of weakly holomorphic modular functions characterized by $J_m(\tau)=q_\tau^{-m}+O(q_\tau).$
Equivalently, for each $m\geq 1$ there is a unique monic polynomial
$F_m$ of degree $m$ such that $F_m(j_1(\tau))=J_m(\tau).$
These polynomials were introduced by Faber, who also gave their
explicit evaluation \cite{Fa1903}.

A remarkable identity of Asai, Kaneko, and Ninomiya \cite{AKN}
shows that the generating function of the family $\{J_m\}_{m\geq 0}$ can be expressed in terms of the $j$-invariant.
Namely, it is proved in \cite{AKN} that
\begin{equation}\label{eq:AKN}
\sum_{m\geq 0}J_m(\tau)q_z^{\,m}
=
-\frac{1}{2\pi i}\frac{j'(z)}{j(z)-j(\tau)}.
\end{equation}
The right-hand side is a meromorphic modular function of weight zero
in $\tau$ and a meromorphic modular form of weight two in $z$. Its
essential analytic feature is that it has a simple pole along the
diagonal $z=\tau$ and no off-diagonal poles. The identity \eqref{eq:AKN} has
played an important role in, among other topics, the study of traces
of singular moduli \cite{Za02} and divisors of modular forms
\cite{Al03,BKO04}.

The study of generating functions for canonical bases of weakly holomorphic modular forms originated primarily in the genus zero setting. For certain congruence subgroups of genus zero, explicit formulas for the generating functions of certain
even-weight canonical bases were established in \cite{DJ08} and \cite{HJ14}. These identities were subsequently extended to a family of genus zero congruence subgroups in \cite{BL15}, and generalized to all genus zero congruence groups via the Borcherds lift in \cite{Ye22}. In \cite{Ye19}, using properties of Green's functions, an analogue of \eqref{eq:AKN} is proved to hold true for the canonical basis of the space of weakly holomorphic modular functions for a generic genus zero Fuchsian group $\Gamma$, with poles supported at the cusp $\infty$ only.
This space is denoted by $M_{0,\Gamma}^{!,\infty}$.

The situation changes substantially in positive genus. If $X_\Gamma$ has positive genus, then its function field can no longer be
generated by a single modular function. In genus one, one may choose
generators $x$ and $y$ satisfying a generalized Weierstrass equation.
The most immediate candidate for an analogue of the
Asai--Kaneko--Ninomiya kernel would then be
\[
-\frac{1}{2\pi i}\frac{x'(z)}{x(z)-x(\tau)}.
\]
This candidate is not satisfactory. Since $x$ has degree two, the map
$x:X_\Gamma\rightarrow\mathbb P^1$ is a double cover, the denominator
vanishes not only at $z=\tau$, but also at the image of $\tau$ under
the corresponding deck involution. Consequently, this naive kernel has an
``unwanted'' off-diagonal pole.

In this paper we prove that the function $\hgz(z,\tau)$ defined for $z, \tau \in \Xg\setminus\{P_\infty\}$, $z\neq \tau$ by
\begin{equation}\label{eq:defnH}
\hgz(z,\tau):=\frac{\bigl(y(\tau)+y(z)+a_1x(z)+a_3\bigr)f(z)}{x(z)-x(\tau)}
\end{equation}
can be viewed as the genus one analogue of the Asai--Kaneko--Ninomiya kernel. Going further, we derive an identity relating $\hgz(z,\tau)$ and the generating function of the Niebur-Poincar\'e series associated to $\Gamma$ (which can also be viewed as an analogue of the Asai--Kaneko--Ninomiya kernel, see Section \ref{sec:intro-wh} below). From this identity, in case when $\Gamma$ is a genus one Atkin-Lenner group,  we deduce the following arithmetic application.

\begin{theoremABC}\label{thmI}
Let $N=p_1\cdots p_r$ be a square-free integer such that the Atkin-Lehner group $\Gamma_0(N)^+$ has genus one. Then the series
\begin{equation}\label{eq. B series}
24\sum_{m\ge1}\sigma(m)\prod_{\nu=1}^{r}\left(1-
\frac{p_\nu^{\alpha_{p_\nu}(m)+1}(p_\nu-1)}{\bigl(p_\nu^{\alpha_{p_\nu}(m)+1}-1\bigr)(p_\nu+1)}\right)q_z^{\,m}
+\frac{3\cdot2^{r}}{\pi\sigma(N)\Im(z)}-1
\end{equation}
is a weight two harmonic Maass form on $\Gamma_0(N)^+$.
\end{theoremABC}
Here $\sigma(m)$ is the sum of the divisors
of a positive integer $m$, and $\alpha_p(m)$ for the largest integer with $p^{\alpha_p(m)}\mid m$. Atkin Lehner groups are defined in Section \ref{sec: arith}.

Inserting $N=1$ formally, so that $r=0$ and the product is empty, turns \eqref{eq. B series} into
\[
24\sum_{m\ge1}\sigma(m)q_z^{\,m}+\frac{3}{\pi\Im(z)}-1
=-\Bigl(E_2(z)-\frac{3}{\pi\Im(z)}\Bigr)=-E_2^{*}(z),
\]
which is the classical weight two harmonic Maass form on $\PSL(2,\ZZ)$. These calculations are strictly formal since $N=1$ is not a genus one level, yet indicate that the Maass form \eqref{eq. B series} can be viewed as a genus one analogue of $E_2^{*}(z)$. Theorem \ref{thmI} is an arithmetic consequence of theorems \ref{thmC} -- \ref{cor:relation H NP} below.

Before stating our main results, we would like to emphasize that positive genus extensions of \eqref{eq:AKN} have previously been obtained in several arithmetic settings. For the hyperelliptic modular curves $X_0(\ell)$ with
$\ell\in\{11,17,19,23,29,31,41,47,59,71\}$, an analogue of the Asai--Kaneko--Ninomiya identity was
established in \cite{E-G09}. Some of these curves are also related to strong Weil curves and to
arithmetic phenomena involving dimensions of vertex operator subalgebras \cite{BM21}. In genus one,
\cite{JM19} studied generating functions for weakly holomorphic row-reduced echelon bases of
$M^{!,\infty}_{k,\Gz}$ for the prime levels $p\in\{11,17,19\}$ and obtained an explicit formula
using modular duality. This was extended in \cite{JKK23} to a broad class of arithmetic Fuchsian
groups generated by congruence subgroups and Atkin--Lehner involutions. In particular,
\cite[Theorem 1.14]{JKK23} expressed the corresponding generating function in terms of basis
elements and generators of the space of weight two cusp forms, while \cite[Theorem 1.16]{JKK23}
related it to Niebur--Poincar\'e series.

Our approach differs from these results in two aspects. First, it begins with the intrinsic
geometry of the elliptic curve and uses a Serre type duality, rather than a Fourier coefficients
duality between modular form bases. Second, it determines the polar structure of the kernel
$\hgz(z,\tau)$ directly, including the cancellation of the off-diagonal pole arising from the
double cover $x\colon\Xg\to\PP^1$. Moreover, our approach does not require any arithmetic input.

\subsection{{Main properties of $\hgz(z,\tau)$}}\label{sec:mainresults}

Let $\Gz\subset\PSL_2(\mathbb R)$ be a cofinite Fuchsian group (possibly with elliptic elements) whose quotient
$\Gz\backslash\HH$ has genus one and exactly one cusp, normalized to lie at $\infty$ with width
one.  Let $\Xg$ denote the smooth compactification of $\Gz\backslash\HH$ with $P_\infty$ the point
above the cusp. We write $q_\tau=e^{2\pi i\tau}$, and $M^{!,\infty}_{0,\Gz}$ for the space of
weakly holomorphic modular functions for $\Gz$ that are holomorphic on $\Xg\setminus\{P_\infty\}$ and have a finite order pole at $P_\infty.$

Let $x$ and $y$ be the normalized generators of the function field of $\Xg$; the normalization is chosen so that
\begin{align}
x(\tau)&=q_\tau^{-2}+s\,q_\tau^{-1}+\sum_{n\ge1}x_nq_\tau^{\,n},\label{eq:xexp}\\
y(\tau)&=q_\tau^{-3}+r\,q_\tau^{-1}+\sum_{n\ge1}y_nq_\tau^{\,n},\label{eq:yexp}
\end{align}
and let $a_1,\dots,a_6$ be the coefficients of the generalized Weierstrass relation \eqref{eq:F}.

Let $f$ be the weight two cusp form for $\Gz$ determined by $\omega=dx/(2y+a_1x+a_3)=-2\pi i
f(\tau)\,d\tau$.  Equivalently, $f$ is normalized so that its first Fourier coefficient at $\infty$
equals $1$, and we write
\begin{equation} \label{eq:f}
f(z)=\sum_{n\ge1}c_nq_z^{\,n}.
\end{equation}
With this notation, our first result is the following theorem.
\begin{theoremABC}\label{thmC}
Let $\tau\in\Xg\setminus\{P_\infty\}$, and let $z$ lie in the neighbourhood $U_\tau$ of $P_\infty$
{determined by $|x(z)|>|x(\tau)|$}\footnote{Since $x(z)\to\infty$ as $z\to P_\infty$, the neighbourhood $U_\tau$ must exist.}. Define $\Phi_m(\tau)$ by the
expansion
\begin{equation}\label{eq:genfn}
\hgz(z,\tau)=\sum_{m\ge0}\Phi_m(\tau)\,q_z^{\,m}.
\end{equation}
Then $\Phi_0\equiv1$ and $\Phi_1\equiv s$.  Furthermore, for every $m\ge2$ the function $\Phi_m$ is a weakly
holomorphic modular function on $\Xg$, holomorphic on $\Xg\setminus\{P_\infty\}$, and is a
polynomial in $x(\tau),y(\tau)$ of degree $m$, with Fourier expansion
\begin{equation}\label{eq:Phiqexp}
\Phi_m(\tau)=q_\tau^{-m}-c_mq_\tau^{-1}+s\,c_m+O(q_\tau).
\end{equation}
\end{theoremABC}
The following theorem shows that $\hgz(z,\tau)$ can be viewed as a reproducing kernel for functions from $ M^{!,\infty}_{0,\Gz}$.
\begin{theoremABC}\label{thmD}
Let $h\in M^{!,\infty}_{0,\Gz}$. Then, for every $\tau\in\Xg\setminus\{P_\infty\}$,
\begin{equation}\label{eq:repro}
h(\tau)=[q_z^{\,0}]\bigl(h(z)\hgz(z,\tau)\bigr).
\end{equation}
Consequently, if $g\in M^{!,\infty}_{0,\Gz}$ has Fourier expansion
$g(\tau)=\sum_{n\ge-N}g_nq_\tau^{\,n}$ for some integer {$N\ge0$}, then
\begin{equation}\label{eq:recovery}
g(\tau)=\sum_{m\ge0}g_{-m}\Phi_m(\tau),
\end{equation}
which is a finite sum with at most $N+1$ terms. In particular $\{\Phi_0\}\cup\{\Phi_m\}_{m\ge2}$ is a
basis of $M^{!,\infty}_{0,\Gz}$, and the coefficients of $g$ satisfy the identity
\begin{equation}\label{eq:gapconstraint}
g_{-1}=-\sum_{m\ge2}c_m\,g_{-m}.
\end{equation}

\end{theoremABC}

\subsection{Comparison between weakly holomorphic and real analytic generalizations of Asai--Kaneko--Ninomiya identity} \label{sec:intro-wh}

According to the Weierstrass gap theorem, if genus of $X_\Gamma$ is positive, it is not possible to have a basis $\{\widetilde J_m\}_{m\geq 1}$ of the space $M_0^{!,\infty}$ of weakly holomorphic functions on $X_\Gamma$ with the poles located only at the cusp, such that all $\widetilde{J}_m$ are weakly holomorphic functions with the $q$-expansion being $q^{-m} + O(1)$, for $m\geq 1$. In other words, a complete analogue of the basis $\{J_m\}_{m\geq 1}$ does not exist in positive genus. For this reason, there are two possible ways to generalize the basis $\{J_m\}_{m\geq 1}$: (i) to retain weak holomorphicity in which case the $q$-expansion must be allowed to include additional term $q^{-1}$, or (ii) to retain the form of the $q$-expansion, in which case the basis is no longer weakly holomorphic but could be, for example, real analytic and harmonic.

The approach taken in this paper is to retain weak holomorphicity, meaning that the identity \eqref{eq:genfn} can be viewed as the weakly holomorphic counterpart of the left-hand side
of \eqref{eq:AKN}. In genus one, the functions $\Phi_m(\tau)$ have a Fourier series expansion \eqref{eq:Phiqexp}
and, because of the presence of the additional $q_\tau^{-1}$, in arithmetic settings do not in general form
a Hecke system. Nevertheless, the weakly holomorphic generalization from Theorem \ref{thmC} preserves the
three central features of the Asai--Kaneko--Ninomiya kernel: it is modular of weight two in the generating variable $z$ and of weight zero in the coefficient variable $\tau$; it has a unique pole on the diagonal; and
its Fourier coefficients generate, and explicitly reconstruct, the space $M^{!,\infty}_{0,\Gz}$.

The second generalization in arithmetic settings preserves a Hecke-system structure at the
cost of replacing weakly holomorphic modular functions by
real-analytic harmonic ones. This construction is provided by the
Niebur--Poincaré series.  Let us now compare these two generating series.

For $m\geq 1$, let $F_{-m}(\tau,s)$ denote the Niebur--Poincaré
series defined in \eqref{Def:Niebur Poinc series}. At $s=1$, the
normalized functions $
j_{\Gamma,m}(\tau)
:=
2\pi\sqrt{m}\,F_{-m}(\tau,1)
$
have Fourier expansions
$
j_{\Gamma,m}(\tau)
=
q_\tau^{-m}+O(1)$, as $\tau\rightarrow\infty$.
They therefore provide a natural real-analytic analogue of $J_m$ on
any finite-volume quotient $\Gamma\backslash\HH$ with a cusp at
$\infty$ of width one. Niebur's results show that these series can be
used to span weakly holomorphic modular functions \cite{Ni73},
although the individual Niebur--Poincaré series are generally
real-analytic rather than weakly holomorphic.

Let us denote the generating function of Niebur-Poincar\'e series by
\begin{equation}\label{eq:def Hcal intro}
\mathcal H_\Gamma(z,\tau)
:=
2\pi\sum_{m\geq 1}
\sqrt{m}\,F_{-m}(\tau,1)q_z^m.
\end{equation}
For the full modular group, \eqref{eq:def Hcal intro} specializes to the
Asai--Kaneko--Ninomiya generating series, up to the classical Eisenstein series $E_2(z)=1-24\sum_{k\geq 1}\sigma(k)q_z^k$.
Specifically, we have that
\[
\mathcal H_{\PSL_2(\ZZ)}(z,\tau)+E_2(z)
=
\sum_{m\geq 0}J_m(\tau)q_z^m.
\]

This real analytic generalization of \eqref{eq:AKN} was developed for congruence groups in
\cite{BK16,BKLOR18}, where the functions
$j_{N,m}(\tau)=2\pi\sqrt{m}F_{-m}(\tau,1)$ were interpreted as a
partial Hecke system of harmonic Maass functions for $\Gamma_0(N)$.
The results were further extended in \cite{BS25}, where an analogue
of \eqref{eq:AKN} was established for congruence groups of level $N$,
and the functions $j_{N,m}$ were shown to form a Hecke system with
respect to the equivariant Hecke operators introduced in
\cite{Ca10,Ca12}.

For general Fuchsian groups of positive genus with parabolic elements, 
\cite[Thm~1.1(2)]{BJS25} established that the mapping
$
\tau\longmapsto\mathcal H_\Gamma(z,\tau)
$
is a weight-zero polar harmonic Maass form, while
$
z\longmapsto
\mathcal H_\Gamma(z,\tau)
+
\frac{1}
{\vol(\Gamma\backslash\HH)\Im(z)}
$
is a weight two polar harmonic Maass form. The latter has a unique
pole in a fundamental domain at $z=\tau$ with polar part
$
\frac{ie_\tau}{2\pi(z-\tau)}.
$

Thus, $\hgz$ and $\mathcal H_\Gamma$ generalize the same type of classical
kernel in complementary ways. The first preserves weak
holomorphicity and yields the basis
$\{\Phi_m\}$. The second preserves the real-analytic Hecke-system
structure. We derive an explicit identity relating these two
generalizations. 

In order to state this relation, we recall some definitions.

For integers $m,n$, let $S(m,n;c)$ denote the Kloosterman sum
associated with $\Gamma$, as defined in \eqref{eq. Kloost sum}. For
$\Re(s)>1$, Selberg's Kloosterman zeta function \cite{Se65, GS83}, is defined by the
absolutely convergent series
\begin{equation}\label{eq:Selberg zeta intro}
L_s(m,n)
:=
\sum_{c>0}S(m,n;c)c^{-2s}.
\end{equation}
where the sum is taken over all positive left lower entries $c$ of matrices from the stabilizer group $\Gamma_\infty$ of the cusp $\infty$.  According to \cite[Theorem 1]{GS83}, the series \eqref{eq:Selberg zeta intro} possesses a meromorphic continuation to $\Re(s)>1/2$ which is holomorphic at $s=1$. We also write $B_0(1;-m)$ for the constant term coefficient in the Fourier expansion \eqref{Four exp Nieb} of
$F_{-m}(\tau,1)$.

With this notation, the relationship between the weakly holomorphic
kernel $\hgz$ and the Niebur--Poincaré generating function $\mathcal H_\Gamma$ is summarized in
the following Theorem.

\begin{theoremABC}
\label{cor:relation H NP}
With the notation above, for $\tau\in X_\Gamma\setminus\{P_\infty\}$ and $z$ in a neighborhood $U_\tau$ of $P_\infty$ one has
\begin{equation}\label{eq. identity H and NP}
\hgz(z,\tau)
=
\mathcal H_\Gamma(z,\tau)
+
f(z)
\bigl(
s+2\pi B_0(1;-1)-2\pi F_{-1}(\tau,1)
\bigr)
-2\pi
\sum_{m\geq 1}
\sqrt{m}\,B_0(1;-m)q_z^m
+1.
\end{equation}
In particular, the following statements hold.

\begin{itemize}
\item[(i)]
The harmonic function
$$
h_\Gamma(z,\tau)
:=
f(z)
\bigl(
s+2\pi B_0(1;-1)-2\pi F_{-1}(\tau,1)
\bigr)
-2\pi
\sum_{m\geq 1}
\sqrt{m}\,B_0(1;-m)q_z^m
+1
$$
completes the polar harmonic Maass form
$\tau\mapsto\mathcal H_\Gamma(z,\tau)$ to the function
$\tau\mapsto \hgz(z,\tau)$, which is holomorphic on
$X_\Gamma\setminus\{P_\infty\}$ except for a simple pole at
$\tau=z$.

\item[(ii)]
The weight two harmonic Maass form
$
z\longmapsto
h_\Gamma(z,\tau)
-
\frac{1}
{\vol(\Gamma\backslash\HH)\Im(z)}
$
completes the function $
z\longmapsto
\mathcal H_\Gamma(z,\tau)
+
\frac{1}
{\vol(\Gamma\backslash\HH)\Im(z)}
$
to the weight two form $z\mapsto \hgz(z,\tau)$, which is
holomorphic on $X_\Gamma$ except for a simple pole at $z=\tau$.

\item[(iii)]
Let
\begin{equation}\label{eq:def B intro}
\mathcal B(z)
:=
4\pi^2
\sum_{m\geq 1}
mL_1(-m,0)q_z^m
\end{equation}
be the generating function of the special values
$\{L_1(-m,0)\}_{m\geq 1}$ of Selberg's Kloosterman zeta function.
Then
\begin{equation}\label{eq:B harmonic intro}
z\longmapsto
\mathcal B(z)-1
+
\frac{1}
{\vol(\Gamma\backslash\HH)\Im(z)}
\end{equation}
is a weight two harmonic Maass form.
\end{itemize}
\end{theoremABC}

We find it interesting that a
sequence of special values of Selberg's Kloosterman zeta function can be encoded in a single modular object which equals holomorphic part of the weight two harmonic Maass form $\mathcal B(z)-1
+
\frac{1}
{\vol(\Gamma\backslash\HH)\Im(z)}$. This shows that the relation between the two generating kernels has
consequences beyond the coefficient-by-coefficient comparison in
\eqref{eq. identity H and NP}. An arithmetic specialization of this result is given in Theorem~\ref{thmI}.

\subsection{Organization of the paper} The paper is organized as follows. In Section \ref{sec:setup} we collect basic facts about the elliptic curve  $\Xg$ and its function field. Polar structure of the kernel $\hgz(z,\tau)$ is studied in Section \ref{sec:kernel}. Proofs of Theorems \ref{thmC} and \ref{thmD} occupy Section \ref{sec:family}. In Section \ref{sec. rel to NP} we prove Theorem \ref{cor:relation H NP} which we specialize to the setting of genus one Atkin-Lehner groups in Section \ref{sec: arith}. In Section \ref{sec: arith} we also present selected numerical calculations in which we record $q_\tau$-expansions of $\Phi_m(\tau)$ and the two-variable polynomials in $x,y$ which yield $\Phi_m$. For the sake of space, calculations are presented for the smallest and the largest genus one levels and for a selected set of $m$. We produced a large set of data covering all genus one levels and the range of $m$ up to $20$.  In addition the program that generates the data can compute all $q$ expansions to arbitrary accuracy and arbitrary values of $m.$ It can be made available upon request.

\section{The elliptic curve $\Xg$ and its function field}\label{sec:setup}
In this section, we collect basic facts regarding the smooth Weierstrass model of $X_\Gamma$ and prove some of its properties needed in the sequel.

\subsection{The smooth Weierstrass model and the deck involution}

Write $\Xg$ for the smooth genus one compactification (the elliptic curve) of $\Gz\backslash\HH$,
and denote the cusp $\infty$ with the local uniformizer $q_\tau=e^{2\pi i\tau}$. Since the genus of
$\Xg$ is $1$, the Weierstrass semigroup of $\Xg$ at $P_\infty$ is $\{0,2,3,4,\dots\}$, with the
single gap at $1$. We assume $\Gz$ to be fixed and omit the index $\Gz$ in the sequel (we keep it
only in the curve notation).

Recall that $x\in L(2P_\infty)\setminus L(P_\infty)$ and $y\in L(3P_\infty)\setminus L(2P_\infty)$ are the
unique functions with $q$-expansions \eqref{eq:xexp}--\eqref{eq:yexp}, fixed by the normalization
\[
[q^{-2}]x=1,\ [q^{0}]x=0;\qquad [q^{-3}]y=1,\ [q^{-2}]y=0,\ [q^{0}]y=0 .
\]
Recall that $s:=[q^{-1}]x$ and $r:=[q^{-1}]y$.

%\subsection{The Weierstrass relation and its coefficients}

The seven functions $1,x,y,x^2,xy,x^3,y^2$ have pole orders $0,2,3,4,5,6,6$ at $P_\infty$ and
therefore lie in $L(6P_\infty)$, a space of dimension $6$. Hence they satisfy the normalized
Weierstrass relation
\begin{equation}\label{eq:F}
F(x,y):=y^2+(a_1x+a_3)y-x^3-a_2x^2-a_4x-a_6=0
\end{equation}
identically on $\Xg$.

The following lemma is a folklore result. For the sake of completeness, we provide a short proof.

%\subsection{Smoothness of the Weierstrass model}

\begin{lemma}\label{lem:smooth}
The plane cubic \eqref{eq:F} is a smooth Weierstrass model of $\Xg$. Equivalently, the
discriminant of \eqref{eq:F} is nonzero. Consequently,
\begin{enumerate}[label=\textup{(\roman*)}]
\item the partials
\[
\frac{\partial F}{\partial y}=2y+a_1x+a_3,\qquad
\frac{\partial F}{\partial x}=a_1y-3x^2-2a_2x-a_4
\]
never vanish simultaneously on $\Xg$, and neither vanishes identically;
\item the map $(x,y)\colon\Xg\to\PP^2$, i.e.\ $[1:x:y]$, is a closed embedding. In particular
$(x,y)$ separates points: if $x(P)=x(P')$ and $y(P)=y(P')$ with $P,P'\ne P_\infty$ then $P=P'$.
\end{enumerate}
\end{lemma}

\begin{proof}
The genus one compact Riemann surface $\Xg$ can be realized as an algebraic curve over $\CC$ in the sense of \cite{Miranda} by \cite[IV.11]{FK}. In addition it has the marked point $P_\infty$, so the pair $(\Xg,P_\infty)$ is an
elliptic curve in the sense of \cite[\S III.3]{Sil}.
By \cite[Prop.~III.3.1(a)]{Sil} and the Riemann-Roch Theorem,  there exist
functions $x'\in L(2P_\infty)$ and $y'\in L(3P_\infty)$ such that $[x':y':1]$ is an
isomorphism of $\Xg$ onto a smooth plane cubic in generalized Weierstrass form. In
particular the discriminant of that model is nonzero, by \cite[Prop.~III.1.4(a)]{Sil}.

Since $\{1,x\}$ and $\{1,x'\}$ are both bases of $L(2P_\infty)$, and $\{1,x,y\}$ and $\{1,x',y'\}$
are both bases of $L(3P_\infty)$, the normalized generators $x,\,y$ given by \eqref{eq:xexp}--\eqref{eq:yexp} are related to $x',y'$ by a linear change of variables
\[
x=u^2x'+t_1,\qquad y=u^3y'+u^2s_1x'+t_2,\qquad u\neq0,
\]
as in \cite[Prop.~III.3.1(b)]{Sil}.  Hence the cubic \eqref{eq:F} is projectively equivalent to
a smooth Weierstrass cubic, is therefore itself smooth with nonzero discriminant, and
$[1:x:y]$ is again an isomorphism onto it.

Assertion (i) is the Jacobian criterion applied to \eqref{eq:F}: a point of the affine chart at
which $\partial F/\partial x$ and $\partial F/\partial y$ both vanish would be a singular point of
the cubic.  Neither partial vanishes identically, each being a nonzero element of $\CC(\Xg)$:
$\partial F/\partial y=2y+a_1x+a_3$ has a pole of order $3$ at $P_\infty$ and
$\partial F/\partial x=a_1y-3x^2-2a_2x-a_4$ a pole of order $4$. Assertion (ii) is the statement
that $[1:x:y]$ is an isomorphism onto a closed subvariety of $\PP^2$; separation of points is
immediate, since $[1:x(P):y(P)]=[1:x(P'):y(P')]$ forces $P=P'$.
\end{proof}

%Alternatively \cite[Prop. VII.1.8]{Miranda} says that \emph{every algebraic curve of genus one is isomorphic to a smooth projective plane curve}. Equation \eqref{eq:F} states one such plane curve.

%\subsection{The sign-change involution}

We recall that the function $x$ realizes $\Xg$ as a double cover $x\colon\Xg\to\PP^1$. We denote by
$\sigma=[-1]$ its deck involution, which is the sign-change map of the elliptic curve.

\begin{lemma}\label{lem:sigma}
Viewing \eqref{eq:F} as a monic quadratic in $y$ over $\CC(x)$, the pull-back $\sigma^\ast$ of the
holomorphic involution $\sigma\colon\Xg\to\Xg$ fixing $P_\infty$ is given by
\begin{equation}\label{eq:sigma}
\sigma^\ast(x,y)=(x,\,-y-a_1x-a_3).
\end{equation}
{The map \eqref{eq:sigma} has exactly four fixed points, namely the $2$-torsion $\Xg[2]$, one of which is $P_\infty$.
At each of the other three points, $\partial F/\partial y$ has a simple zero.}
\end{lemma}

\begin{proof}
View \eqref{eq:F} as a monic quadratic in $y$ over $\CC(x)$. Its two roots $y,y'$ satisfy
$y+y'=-(a_1x+a_3)$, so the nontrivial deck transformation sends $y\mapsto-y-a_1x-a_3$ and
fixes $x$, which defines $\sigma$. Its fixed points are where the two sheets meet, i.e.\ where
$\partial F/\partial y=2y+a_1x+a_3=0$. Such points are the branch points of $x$, which are the $2$-torsion of
$(\Xg,P_\infty)$. By \cite[Prop. III.7.9]{FK} $x$ has exactly four fixed points. Since $x$ only has a pole at $P_\infty$ one of the fixed points of $\sigma$ is $P_\infty.$ Since $\partial F/\partial y$ has only a pole at $P_\infty,$ the degree of its zero divisor must also be three, so then each zero must be simple.
\end{proof}

\subsection{The holomorphic differential and the weight two form}
We define
\begin{equation}\label{eq:ggvee}
g:=\frac{\partial F}{\partial y}=2y+a_1x+a_3,\qquad
g^{\vee}:=\frac{\partial F}{\partial x}=a_1y-3x^2-2a_2x-a_4.
\end{equation}
Differentiating \eqref{eq:F} gives $g\,dy-(-g^{\vee})\,dx=0$, whence the invariant differential
\begin{equation}\label{eq:omega}
\omega:=\frac{dx}{2y+a_1x+a_3}=\frac{dy}{3x^2+2a_2x+a_4-a_1y}
=\frac{dx}{g}=\frac{dy}{-g^{\vee}}
\end{equation}
is a global holomorphic $1$-form on $\Xg$. Lemma \ref{lem:smooth} guarantees that at every
point at least one of the two expressions is regular and nonzero. Since
$\dim H^0(\Xg,\Omega^1)=1$, the form $\omega$ spans the holomorphic differentials, is holomorphic and  non-vanishing;  see \cite[Prop.~III.1.5]{Sil}.

The weight two automorphic form $f$ for $\Gz$ is defined by
\begin{equation}\label{eq:fdef}
\omega=-2\pi i\,f(\tau)\,d\tau=-f(\tau)\frac{dq_\tau}{q_\tau},
\end{equation}
the second equality holding because $q_\tau=e^{2\pi i\tau}$. Then $f$ is a cusp form, and we write
\begin{equation}\label{eq: f exp}
f(\tau)=\sum_{n\ge1}c_nq_\tau^{\,n}.
\end{equation}
From \eqref{eq:omega} and \eqref{eq:fdef}, and by using the chain rule, we have that
\begin{equation}\label{eq:recursionQ}
f(\tau)\bigl(2y(\tau)+a_1x(\tau)+a_3\bigr)=-q_\tau\frac{dx}{dq_\tau}.
\end{equation}

Some of our residue calculations become complicated at the ramification points of $x.$ We need the following lemma which says that $x$ and $y$ never ramify at the same point of $\Xg.$

\begin{lemma}\label{lem:dichotomy}
For every $P\in\Xg\setminus\{P_\infty\}$,
\[
\ord_P\bigl(x-x(P)\bigr)=1+\ord_P(g),\qquad
\ord_P\bigl(y-y(P)\bigr)=1+\ord_P(g^{\vee}),
\]
and consequently
\[
\min\Bigl\{\ord_P\bigl(x-x(P)\bigr),\ \ord_P\bigl(y-y(P)\bigr)\Bigr\}=1 .
\]
In particular, at least one of the functions $x-x(P)$ and $y-y(P)$ is a local uniformizer at $P$.
In particular, if $\sigma(P)=P$, so that $g(P)=0$, then $y-y(P)$ is a local uniformizer at $P$.
\end{lemma}

\begin{proof}
By \eqref{eq:omega} we have $dx=g\,\omega$ and $dy=-g^{\vee}\omega$. Since $\omega$ is a nowhere
vanishing holomorphic $1$-form on $\Xg$, the order of vanishing at $P$ of the differential $dx$
equals $\ord_P(g)$, and $\ord_P(dx)=\ord_P(x-x(P))-1$; the same for $y$ with $g^{\vee}$. The last
assertion of Lemma~\ref{lem:smooth}(i) is that $g$ and $g^{\vee}$ have no common zero on $\Xg$,
which is precisely the Jacobian criterion for \eqref{eq:F}.  Hence, at every $P$ at least one of
$\ord_P(g),\ord_P(g^{\vee})$ is zero. Finally, $g(P)=0$ characterizes the fixed points of $\sigma$
by Lemma~\ref{lem:sigma}, and there $g^{\vee}(P)\ne0$.
\end{proof}

\begin{lemma}\label{lem:coeffs}
With the expansions \eqref{eq:xexp}--\eqref{eq:yexp} one has
\begin{equation}\label{eq:a1eq3s}
a_1=3s,\qquad a_2=2r,\qquad a_3=3x_1+s^3+rs,
\end{equation}
and the first Fourier coefficients of $f$ satisfy
\begin{equation}\label{eq:c1c2}
c_1=1,\qquad c_2=-s,\qquad c_3=-r,\qquad c_4=-2x_1+s^3+2sr .
\end{equation}
\end{lemma}

\begin{proof}
Existence of the relation \eqref{eq:F} is the dimension count above.  The difference of the two sides
is a holomorphic function on the compact connected surface $\Xg$ once the principal parts cancel;
hence the difference is a constant, and the constant is absorbed into $a_6$. The identities \eqref{eq:a1eq3s} follow by
computing the coefficients of $q^{-5},q^{-4},q^{-3}$ in \eqref{eq:F}. For instance, at $q^{-5}$,
\[
[q^{-5}]y^2=0,\quad [q^{-5}](a_1xy)=a_1,\quad [q^{-5}]x^3=3s,\quad [q^{-5}](a_2x^2)=0,
\]
and all remaining terms have order $>-5$, so the vanishing of $[q^{-5}]F$ reads $a_1-3s=0.$ The other computations are similar. The identities \eqref{eq:c1c2} follow by comparing $q_\tau$-expansions on both sides of \eqref{eq:recursionQ}. \end{proof}

\section{Polar structure of the kernel $\hgz(z,\tau)$}\label{sec:kernel}

Using \eqref{eq:omega} and \eqref{eq:fdef}, the kernel \eqref{eq:defnH} can be written as a
difference, namely
\begin{equation}\label{eq:Hdiff}
\hgz(z,\tau)=-\frac{1}{2\pi i}\frac{x'(z)}{x(z)-x(\tau)}
+\frac{y(\tau)-y(z)}{x(z)-x(\tau)}f(z),
\end{equation}
which exhibits it as the naive Asai--Kaneko--Ninomiya candidate corrected by a term that removes the
off-diagonal pole. It also admits a third expression, with $y$ rather than $x$ in the
denominator, which is what makes the polar analysis uniform.

\begin{lemma}\label{lem:twoforms}
Write $N(z,\tau):=y(\tau)+y(z)+a_1x(z)+a_3$ for the numerator of \eqref{eq:defnH} and set
\[
M(z,\tau):=x(z)^2+x(z)x(\tau)+x(\tau)^2+a_2\bigl(x(z)+x(\tau)\bigr)+a_4-a_1y(\tau).
\]
Then, identically on $\Xg\times\Xg$,
\begin{equation}\label{eq:NM}
N(z,\tau)\bigl(y(z)-y(\tau)\bigr)=\bigl(x(z)-x(\tau)\bigr)M(z,\tau),
\end{equation}
and consequently
\begin{equation}\label{eq:yform}
\hgz(z,\tau)=\frac{N(z,\tau)f(z)}{x(z)-x(\tau)}=\frac{M(z,\tau)f(z)}{y(z)-y(\tau)} .
\end{equation}
Moreover, on the diagonal,
\begin{equation}\label{eq:diag}
N(\tau,\tau)=g(\tau),\qquad M(\tau,\tau)=-g^{\vee}(\tau).
\end{equation}
\end{lemma}

\begin{proof}
If one expands \eqref{eq:NM} and subtracts the right side from the left, it simplifies to a difference of
Weierstrass relations \eqref{eq:F}, meaning
$$
F(x(\tau),y(\tau)) - F(x(z),y(z)) = 0-0.
$$
Setting
$z=\tau$ in $N$ and in $M$ gives \eqref{eq:diag} directly from \eqref{eq:ggvee}.
\end{proof}

We define the meromorphic one-form
\begin{equation}
\label{eqOneForm}
{\eta_\lambda:=-2\pi i\,\hgz(z,\lambda)\,dz}
\end{equation}

\begin{lemma} \label{lem:A}
Let $\lambda\in\Xg\setminus\{P_\infty\}$. Then $z\mapsto \hgz(z,\lambda)$ is a weight two
meromorphic modular form for $\Gz$ whose only pole on $\Xg$ is a simple pole at $z=\lambda$, with
residue $-e_\lambda/2\pi i$ in the upper half-plane coordinate, where $e_\lambda=|\Stab_\lambda|$ is the order of the stabilizer. The meromorphic one-form $\eta_\lambda$ is holomorphic on $\Xg\setminus\{\lambda,P_\infty\}$, with only simple poles at $\lambda$ and $P_\infty;$ with
$\Res_{z=\lambda}\eta_\lambda=1$ and $\Res_{z=P_\infty}\eta_\lambda=-1$.
\end{lemma}

\begin{proof}
Recall $$\hgz(z,\tau):=\frac{\bigl(y(\tau)+y(z)+a_1x(z)+a_3\bigr)f(z)}{x(z)-x(\tau)}$$

Fix $\lambda\in\Xg\setminus\{P_\infty\}$. Since $y$ and $f$ are holomorphic on
$\Xg\setminus\{P_\infty\}$, the only possible singularities of $z\mapsto \hgz(z,\lambda)$ on
$\Xg\setminus\{P_\infty\}$ are at the points where $x(z)-x(\lambda)=0$; that is, at $z=\lambda$ and,
when $\sigma(\lambda)\ne\lambda$, at $z=\sigma(\lambda)$. We treat these in turn, then the cusp.

\medskip
%\noindent\emph{Step 1: The off-diagonal pole $z=\sigma(\lambda)$ is removable.}
Suppose $\sigma(\lambda)\ne\lambda$. At $z=\sigma(\lambda)$ the denominator
$x(z)-x(\lambda)=x(\sigma(\lambda))-x(\lambda)=0$ vanishes to order one, because $\sigma(\lambda)$ is
not a ramification point of the degree-two map $x$. The numerator also vanishes there, because by
\eqref{eq:sigma}
\[
y(\lambda)+y(\sigma(\lambda))+a_1x(\sigma(\lambda))+a_3
=y(\lambda)+\bigl(-y(\lambda)-a_1x(\lambda)-a_3\bigr)+a_1x(\lambda)+a_3=0 .
\]
Since the numerator vanishes to order at least one and the denominator to order exactly one, the
function $z\mapsto \hgz(z,\lambda)$ does not have a pole at $z=\sigma(\lambda)$.

\medskip
%\noindent\emph{{Step 2: The diagonal point $z=\lambda$, by two charts.}}
Let $t$ be a local coordinate on $\Xg$ at $\lambda$ and write $\omega=h(t)\,dt$ with $h(0)\ne0$,
which is possible because $\omega$ is nowhere vanishing. By \eqref{eq:fdef},
$f(z)\,dz=-\omega/(2\pi i)$.

Suppose first that $g(\lambda)\ne0$, equivalently $\sigma(\lambda)\ne\lambda$. By
Lemma~\ref{lem:dichotomy}, $x-x(\lambda)$ is a local uniformizer at $\lambda$, and by
\eqref{eq:omega} we have $dx = g \omega$. This implies the Taylor series expansion that
$x(z)-x(\lambda)=g(\lambda)h(0)t+O(t^2)$. By \eqref{eq:diag} the numerator
$N(z,\lambda)$ tends to $g(\lambda)\ne0$. Hence, from the $x$-form of \eqref{eq:yform},
\[
\hgz(z,\lambda)\,dz=\frac{g(\lambda)+O(t)}{g(\lambda)h(0)t+O(t^2)}\cdot\frac{-h(t)\,dt}{2\pi i}
=-\frac{1}{2\pi i}\frac{dt}{t}+O(1)\,dt .
\]

Suppose next that $g(\lambda)=0$, or, equivalently, $\sigma(\lambda)=\lambda$. By
Lemma~\ref{lem:smooth}(i) we then have $g^{\vee}(\lambda)\ne0$, so by Lemma~\ref{lem:dichotomy} the
function $y-y(\lambda)$ is a local uniformizer at $\lambda$, and by \eqref{eq:omega} we have
$y(z)-y(\lambda)=-g^{\vee}(\lambda)h(0)t+O(t^2)$. By \eqref{eq:diag} the numerator $M(z,\lambda)$
tends to $-g^{\vee}(\lambda)\ne0$. Hence, from the $y$-form of \eqref{eq:yform},
\[
\hgz(z,\lambda)\,dz=\frac{-g^{\vee}(\lambda)+O(t)}{-g^{\vee}(\lambda)h(0)t+O(t^2)}
\cdot\frac{-h(t)\,dt}{2\pi i}=-\frac{1}{2\pi i}\frac{dt}{t}+O(1)\,dt .
\]

In either case $\hgz(z,\lambda)\,dz$ has a simple pole at $z=\lambda$ with residue $-1/2\pi i$ in
the local coordinate on $\Xg$. If $t$ is related to the upper half-plane coordinate $u$ at $\lambda$
by $t=u^{e_\lambda}$, then $dt/t =\,e_\lambda\,du/u$, so the residue in the upper half-plane
coordinate is $-e_\lambda/2\pi i$, as asserted. In particular $\Res_{z=\lambda}\eta_\lambda=1$ for
$\eta_\lambda=-2\pi i\hgz(z,\lambda)\,dz$.

\medskip
%\noindent\emph{Step 3: The cusp $P_\infty$ is a removable singularity of $\hgz$.}
By \eqref{eq:xexp}--\eqref{eq:yexp}, in a neighbourhood of the cusp the numerator of
\eqref{eq:defnH} has leading term $q_z^{-3}$ and the denominator $x(z)-x(\lambda)$ has leading term
$q_z^{-2}$. Multiplying by $f=q_z+O(q_z^2)$, whose simple zero cancels the residual pole, gives
$\hgz\sim q_z^{0}$ with leading coefficient $1\cdot1\cdot1=1$. Using \eqref{eq:xexp},
\eqref{eq:yexp}, \eqref{eq: f exp}, \eqref{eq:c1c2} and \eqref{eq:a1eq3s} one deduces the first
terms of the $q_z$-expansion:
\begin{align}
\hgz(z,\tau)=1&+s\,q_z+\bigl(x(\tau)+r+c_3-s^2\bigr)q_z^{2}\notag\\
&+\bigl(y(\tau)+a_3+c_4-x_1-2s^3+2sc_3-2sr\bigr)q_z^{3}+O(q_z^{4}).\label{eq:Hexp}
\end{align}
By Lemma \ref{lem:coeffs} these collapse, by which we mean the following. Since $c_3=-r$, the $q_z^2$ coefficient is
$x(\tau)-s^2$, and upon substituting $a_3=3x_1+s^3+rs$ and $c_4=-2x_1+s^3+2sr$, the $q_z^3$ coefficient
is $y(\tau)-sr$. Thus
\begin{equation}\label{eq:Hexpsimp}
{\hgz(z,\tau)=1+s\,q_z+\bigl(x(\tau)-s^2\bigr)q_z^{2}+\bigl(y(\tau)-sr\bigr)q_z^{3}+O(q_z^{4}),}
\end{equation}
so then $\hgz$ has a removable singularity at $P_{\infty}$.

\medskip
%\noindent\emph{Step 4: Summary.}
Since $z\mapsto \hgz(z,\lambda)$ is a rational function in the modular functions $x(z)$ and $y(z)$ times
the weight two cusp form $f(z)$, it is a weight two meromorphic modular form for $\Gz$.  From the above analysis, we conclude that its only pole on $\Xg$ is the simple pole at $z=\lambda$.

Finally we record the residue at $P_\infty$ of the meromorphic one-form  $\eta_\lambda$. Since
$dz=dq_z/(2\pi i\,q_z)$,
\[
\hgz(z,\lambda)\,dz=\Bigl(\sum_{m\ge0} A_m(\lambda)q_z^{\,m}\Bigr)\frac{dq_z}{2\pi i\,q_z},
\qquad\text{so}\qquad
\Res_{z=P_\infty}\hgz(z,\lambda)\,dz=\frac{ A_0(\lambda)}{2\pi i},
\]
whence $\Res_{z=P_\infty}\eta_\lambda=-A_0(\lambda)=-1$ by \eqref{eq:Hexp}.

\end{proof}

\begin{lemma} \label{thmB}
Let $z\in\mathbb H$ be point\footnote{we identify $z$ with its projection in $\Xg.$} that is not an elliptic fixed point of $\Gamma;$ equivalently, $f(z)\neq 0$. Then $\tau\mapsto \hgz(z,\tau)$ is a meromorphic modular function on $\Xg$ with
precisely two poles, both simple, at $\tau=z$ and at $\tau=P_\infty$.  Furthermore,
\[
\hgz(z,\tau)=-f(z)\,q_\tau^{-1}+O_z(1)\qquad\text{as }\tau\to P_\infty .
\]
If $f(z)=0$ then $\tau\mapsto \hgz(z,\tau)$ vanishes identically.
\end{lemma}

\begin{proof}

The function $\tau\mapsto \hgz(z,\tau)$ is a rational function in $x(\tau)$ and $y(\tau)$, hence
$\Gz$-invariant and meromorphic on $\Xg$.

Suppose first that $f(z)=0$. By \eqref{eq:defnH}, the kernel then vanishes identically in
$\tau$, and there is nothing more to prove. This occurs at elliptic fixed points. By \eqref{eq:omega} and
\eqref{eq:fdef}, $\omega$ is nowhere vanishing on $\Xg$ while the covering map $\HH\to\Gz\backslash\HH$
ramifies to order $e_z$ at an elliptic fixed point $z$, so $f$ vanishes to order $e_z-1$ there.

Assume now $f(z)\ne0$. In the variable $\tau$ the denominator of \eqref{eq:defnH} vanishes at
$\tau=z$ and at $\tau=\sigma(z)$. At the latter the numerator vanishes as well.  Reasoning analogously as in the proof of Lemma~\ref{lem:A}, when using \eqref{eq:sigma}, we have that
\[
y(\sigma(z))+y(z)+a_1x(z)+a_3=\bigl(-y(z)-a_1x(z)-a_3\bigr)+y(z)+a_1x(z)+a_3=0.
\]
So, if $\sigma(z)\ne z$, the singularity at $\tau=\sigma(z)$ is removable. If $\sigma(z)=z$ the two
points coincide and analysis conducted in the proof of  Lemma~\ref{lem:A} applies verbatim with the roles of the variables
exchanged, using the $y$-form \eqref{eq:yform}. In either case the only pole of
$\tau\mapsto \hgz(z,\tau)$ on $\Xg\setminus\{P_\infty\}$ is a simple pole at $\tau=z$.

At $\tau=P_\infty$, the $q_\tau$-expansions \eqref{eq:xexp}--\eqref{eq:yexp} give
$N(z,\tau)=q_\tau^{-3}+O(q_\tau^{-1})$ and $x(z)-x(\tau)=-q_\tau^{-2}+O(q_\tau^{-1})$, whence
$\hgz(z,\tau)=-f(z)q_\tau^{-1}+O_z(1)$, and the pole at $P_\infty$ is simple with residue
$-f(z)/2\pi i$ in the local coordinate $q_\tau$. Since $f(z)\ne0$ this pole exists with non-zero residue.

\end{proof}

\section{Proof of Theorems \ref{thmC} and \ref{thmD}}\label{sec:family}

In this section we prove Theorems \ref{thmC} and \ref{thmD}. Recall that $U_\tau$ is the neighborhood of $P_\infty$ in the $z$-variable such that $|x(z)|>|x(\tau)|$ We start with the following lemma.

\begin{lemma}\label{lem:denom}
Let $\tau\in\Xg\setminus\{P_\infty\}$. For $z\in U_\tau$ one has that
\[
\frac{1}{x(z)-x(\tau)}=q_z^{2}\sum_{k\ge0}p_k(x(\tau))\,q_z^{\,k},
\]
where $p_k(x(\tau))$ is a polynomial in $x(\tau)$ of degree $\lfloor k/2\rfloor$. The
polynomial $p_{k}$ is monic when $k$ is
even. Consequently, the expansion \eqref{eq:genfn} converges normally on  $U_\tau.$
\end{lemma}

\begin{proof}
Write $q=q_z$, and let
\begin{equation}
  x(z) \;=\; q^{-2} + s\,q^{-1} + \sum_{n\ge1} x_n q^{n}, \qquad A:=x(\tau),
\end{equation}

When factoring out  $q^{-2}$, we get that
$
  x(z)-A \;=\; q^{-2}\Bigl(1 + s\,q - A\,q^{2} + \sum_{n\ge1}x_n q^{\,n+2}\Bigr)
          \;=:\; q^{-2}\,u(q).
$ Then $u(q)$ is a unit, and in particular we have that
$
  \frac{1}{x(z)-A} \;=\; q^{2}\,u(q)^{-1}.
$
Write $u=1+w$ so then
$$
  w(q) \;=\; s\,q \;-\; A\,q^{2} \;+\; x_1q^{3} + x_2q^{4} + x_3q^{5}+\cdots
  \;=\; \sum_{j\ge1} u_j q^{j},
$$
so $u_1=s$, $u_2=-A$, and $u_j=x_{j-2}$ for $j\ge3$.
 Next, by expanding a geometric series, we have that
$u^{-1} \;=\; \sum_{m\ge0}(-1)^m w^{m}.
$ Comparing coefficients in $u\cdot u^{-1}=1$ with
$u^{-1}=\sum_kp_k(A)q^k$
gives $p_0(A)=1$ and $\sum_{j=0}^{k}u_jp_{k-j}(A)=0$ for $k\ge1$ (with $u_0=1$),
\begin{equation}
  \;p_k(A)\;=\;A\,p_{k-2}(A)\;-\;s\,p_{k-1}(A)\;-\;\sum_{i=1}^{k-2}x_i\,p_{k-2-i}(A),\qquad p_0=1,\;
  \label{eq:recursion}
\end{equation}
where $p_0(A)=1$ and $p_{k}(A)=0$ for $k < 0.$ The order of $p_k(A)$ follows by induction from \eqref{eq:recursion}

Since $|x(\tau)/x(z)| < 1$ the geometric series for $\frac{1}{x(z)-x(\tau)}$ converges normally and is
 holomorphic in $z$ in $U_{\tau}$.
\end{proof}

Now, we turn to proof of Theorem \ref{thmC}. Recall the definition \eqref{eq:genfn} of $\{\Phi_m\}_{m\geq 0}$, take $\tau\in\Xg\setminus\{P_\infty\}$ and let $z\in U_\tau$. Recall \eqref{eq:defnH} and set $u = q_z$.  By Lemma~\ref{lem:denom}, on $U_\tau$,
\begin{equation}
  \frac{1}{x(z)-x(\tau)}\;=\;u^{2}\sum_{k\ge0}p_k\bigl(x(\tau)\bigr)u^{k},
  \label{eq:den}
\end{equation}
with $p_k$ a polynomial in $x(\tau)$ of degree $\lfloor k/2\rfloor$, monic for $k$ even.
Using \eqref{eq:xexp} and \eqref{eq:yexp} we have that
\begin{equation}
  y(z)+a_1x(z)+a_3\;=\;\sum_{j\ge-3}b_ju^{j}.
  \label{eq:bdef}
\end{equation}
Recalling that $a_1=3s$, it follows that
\begin{equation}
  b_{-3}=1,\qquad b_{-2}=a_1=3s,\qquad b_{-1}=r+a_1s=r+3s^{2},\qquad b_{0}=a_3,\qquad b_{j}=y_j+3s\,x_j\ (j\ge1).
\label{eq:bvals}
\end{equation}
Using expression \eqref{eq:yform} we then have that
\begin{equation} \label{eqHexpanded3}
  \hgz(z,\tau)=\sum_{k\ge0}p_k(x(\tau))u^{k} \cdot
\Bigl(1+3su+(r+3s^{2})u^{2}+(a_3+y(\tau))u^{3}+\textstyle\sum_{j\ge1}(y_j+3sx_j)u^{j+3}\Bigr)\cdot
  \sum_{\ell\ge1}c_\ell u^{\ell-1}.
\end{equation}

All series on the right-hand side of \eqref{eqHexpanded3} converge absolutely and uniformly. After multiplication, using the
$u= q_z$-expansion \eqref{eq:xexp} of $x(z)$, we deduce that the coefficients $\Phi_m(\tau)$ are
polynomials in $x(\tau)$ and $y(\tau)$, hence holomorphic modular functions on
$\Xg\setminus\{P_\infty\}$. The assertion that $\Phi_0\equiv1$ and $\Phi_1\equiv s$ was established in
\eqref{eq:Hexp}.

We next determine the pole orders. When $m=2\ell\ge2$ is even, it is evident from \eqref{eqHexpanded3}
that the term with the lowest power of $q_\tau$ in the $q_\tau$-expansion of $\Phi_m(\tau)$ is
$p_m(x(\tau))$, a monic polynomial of degree $\lfloor m/2\rfloor=\ell$ in $x(\tau)$.  In view of
\eqref{eq:xexp}, the lead term is therefore $q_\tau^{-m}$. When $m>2$ is odd, write $m=2\ell+3$; the
lead term now stems from the product of $p_{2\ell}(x(\tau))$ and $y(\tau)$ and equals
$q_\tau^{-(2\ell+3)}=q_\tau^{-m}$. This proves the first assertion of Theorem \ref{thmC} and the statement that
\begin{equation} \label{eq:qorder}
\Phi_m(\tau)=q_\tau^{-m}+O\bigl(q_\tau^{-(m-1)}\bigr).
\end{equation}

In order to complete the proof of Theorem \ref{thmC}, we need to prove the expansion \eqref{eq:Phiqexp}. In order to do so, we first prove that the kernel $\hgz(z,\tau)$ can be viewed as the reproducing kernel for $ M^{!,\infty}_{0,\Gz}$, which is the first part of Theorem \ref{thmD}.
Namely, we claim that, for $h\in M^{!,\infty}_{0,\Gz}$ with $q$-expansion
$h(z)=\sum_{n\ge-N}h_nq_z^{\,n}$ equation \eqref{eq:repro} holds true.

Define the meromorphic $1$-form
\[
\phi(z) :=-2\pi i\,h(z)\hgz(z,\tau)\,dz\qquad\text{on }\Xg .
\]
By Lemma~\ref{lem:A} the only poles of $\phi$ in the $z$-variable on $\Xg$ are at $z=\tau$ and at
the cusp $z=P_\infty$.

The residue theorem on the compact Riemann surface\footnote{On the upper half-plane coordinate, the relation is
$t=u^{e_\tau}$ and hence $dt/t= e_\tau\,du/u$: the half-plane residue is $e_\tau$ times the
residue on $\Xg.$} $\Xg$ gives
$\sum_{P\in\Xg}\Res_P(\phi)=0.$  By Lemma~\ref{lem:A},
$\Res_{z=\tau}\hgz(z,\tau)\,dz=-1/2\pi i$, so
\[
\Res_{z=\tau}(\phi)=-2\pi i\,h(\tau)\cdot\Bigl(-\frac{1}{2\pi i}\Bigr)=h(\tau).
\]

At the cusp we have $\phi=-h(z)\hgz(z,\tau)\,dq_z/q_z$, and since $\hgz$ is holomorphic at $q_z=0$,
\[
\Res_{z=P_\infty}(\phi)=-[q_z^{0}]\bigl(h(z)\hgz(z,\tau)\bigr).
\]
Summing the two residues yields \eqref{eq:repro}.

\medskip
Now we prove the second statement of Theorem  \ref{thmD}. Let $g\in M^{!,\infty}_{0,\Gz}$ with Fourier expansion $g(\tau)=\sum_{n\ge-N}g_nq_\tau^{\,n}$ for some integer {$N\ge0$}. We apply \eqref{eq:repro} with $h=g$ and expand $\hgz$ by \eqref{eq:genfn} to get
$$
g(\tau)=[q_z^{\,0}]\left(\left( \sum_{n\ge-N}g_nq_z^{\,n} \right)   \left(\sum_{m\ge0}\Phi_m(\tau)\,q_z^{\,m}\right) \right).
$$
By reading off the constant term, we conclude that $g(\tau)=\sum_{m\ge0}g_{-m}\Phi_m(\tau)$.
The sum is finite because $g_{-m}=0$ for \mbox{$m>N$.} This proves \eqref{eq:recovery}.

Now, we complete the proof of Theorem  \ref{thmC} by proving the expansion \eqref{eq:Phiqexp}.
\medskip
For $m\ge2$ write $\Phi_m(\tau)=\sum_{k\ge-m}a^{(m)}_kq_\tau^{\,k}$. Then $a^{(m)}_{-m}=1$, by \eqref{eq:qorder}. We claim that $a^{(m)}_{-k}=0$ for $2\le k\le m-1$, and
\[
\delta_m:=a^{(m)}_{-1}=-c_m,\qquad \epsilon_m:=a^{(m)}_{0}=s\,c_m,
\qquad\text{hence}\qquad \epsilon_m=-s\,\delta_m .
\]

Applying \eqref{eq:repro} with $h=\Phi_m$ and \eqref{eq:recovery} with $g=\Phi_m$, yields that
\[
\Phi_m(\tau)=a^{(m)}_{-m}\Phi_m(\tau)+a^{(m)}_{-m+1}\Phi_{m-1}(\tau)+\cdots+a^{(m)}_{-1}\Phi_1(\tau)+a^{(m)}_{0}\Phi_0(\tau).
\]
This holds for all $\tau\in\Xg\setminus\{P_\infty\}$. The functions $\Phi_0,\Phi_2,\Phi_3,\dots,
\Phi_{m-1}$ are linearly independent, having pairwise distinct pole orders at $P_\infty$ together
with the constant function.  Hence, upon comparing the two sides and using $a^{(m)}_{-m}=1$, we get that
$a^{(m)}_{-k}=0$ for $2\le k\le m-1$.  Since $\Phi_1\equiv s$ and $\Phi_0\equiv1$, we have
\[
a^{(m)}_{-1}s+a^{(m)}_{0}=0 .
\]
For the remaining identity we use Serre duality. For $m\ge2$ the product $\Phi_m\omega$ is a
meromorphic $1$-form on $\Xg$ whose only pole is at $P_\infty$, so
$\Res_{\tau=P_\infty}(\Phi_m\omega)=0$. Since $\omega=-f(\tau)\,dq_\tau/q_\tau$ by
\eqref{eq:fdef}, this reads $[q_\tau^{0}]\bigl(\Phi_m(\tau)f(\tau)\bigr)=0$, that is
$\sum_{n=1}^{m}c_na^{(m)}_{-n}=c_1a^{(m)}_{-1}+c_ma^{(m)}_{-m}=0$, where we used
$a^{(m)}_{-n}=0$ for $2\le n\le m-1$. Since $c_1=1$ and $a^{(m)}_{-m}=1$ we get $a^{(m)}_{-1}=-c_m$,
and then $a^{(m)}_{0}=sc_m$.
This completes the proof of Theorem \ref{thmC}.
\medskip

It is left to prove the last statement of Theorem \ref{thmD}. Let $g$ have principal part
$g_{-N}q_\tau^{-N}+\cdots+g_{-1}q_\tau^{-1}+g_0.$ Reading off the coefficients of $q_\tau^{-1}$ and on both sides of \eqref{eq:recovery} and using expansion \eqref{eq:Phiqexp} gives
\[
g_{-1}=\sum_{m\ge2}\delta_m\,g_{-m}=-\sum_{m\ge2}c_m\,g_{-m}.
\]

\section{Comparison of weakly holomorphic and real analytic generalizations}\label{sec. rel to NP}

In this section we prove Theorem~\ref{cor:relation H NP}. We start by recalling known results in section \ref{sec np prel}. The proof occupies Section \ref{sec proof D}.

\subsection{Preliminaries} \label{sec np prel}

We recall that a harmonic Maass form of weight $k\in \mathbb{Z}$ for $\Gamma$ is a real-analytic function on $\mathbb H$ that transforms
modular of weight $k$ under $\Gamma$, has at most linear exponential growth at the cusp, and is annihilated by the weighted Laplaciandefined for $z=x+iy\in\HH$ by
\begin{equation*}
\Delta_{k,z}:=-y^2\Bigl(\frac{\partial^2}{\partial x^2}+\frac{\partial^2}{\partial y^2}\Bigr)
+ikv\Bigl(\frac{\partial}{\partial x}+i\frac{\partial}{\partial y}\Bigr).
\end{equation*}
A polar harmonic Maass form is allowed to have isolated singularities of finite order in
$\mathbb H$.  Away from its singularities a polar harmonic Maass form is smooth and annihilated by $\Delta_k$. The local singular behavior is understood in an elliptic local coordinate on the quotient (see the proof of Lemma~\ref{lem:A}).

The Niebur-Poincar\'e series $F_{m}(z,s)$ is defined for $m\in\mathbb Z\setminus \{0\}$, $z\in\HH$  and  $s\in\CC$ with $\Re(s)>1$ by the series 
\begin{equation}\label{Def:Niebur Poinc series}
  F_m(z,s)=F_m^\Gamma(z,s):= \sum_{\gamma\in\Gamma_\infty \backslash \Gamma} e(m\Re(\gamma z))(\Im(\gamma z))^{1/2}I_{s-1/2}(2\pi |m| \Im(\gamma z)),
\end{equation}
which converges absolutely and uniformly on any compact subset of the half plane
$\Re(s)>1$. Here $I_{s-1/2}$ denotes the modified $I$-Bessel function of the first kind and $\Gamma_\infty$ is the stabilizer of the cusp $\infty$.

The function $F_m(z,s)$ admits a meromorphic continuation to the whole complex plane $s\in\CC$ \cite{Ni73}, holomorphic at $s=1$. The Fourier expansion of $F_{m}(z,s)$ in the cusp $\infty$ is derived in \cite{Ni73} and involves the Kloosterman sums $S(m,n;c)$, which we now define.  For
any $m,\, n\in\mathbb Z$, and a real number $c\neq 0$, define
\begin{equation}\label{eq. Kloost sum}
S(m,n;c):= \sum_{\bigl(\begin{smallmatrix}
a&\ast\\c&d\end{smallmatrix}\bigr)\in \Gamma_\infty \backslash \Gamma / \Gamma_\infty } e\left( \frac{ma + nd}{c}\right),
\end{equation}
where $e(x)=\exp(2\pi i x)$. With this notation, the Fourier expansion of $F_m(z,s)$, for   $\Re(s)>1$ and $z=x+iy\in \HH$, is given by
\begin{equation}\label{Four exp Nieb}
F_m(z,s)=e(mx)y^{1/2}I_{s-1/2}(2\pi |m|y) + \sum_{k=-\infty}^{\infty}b_k(y,s;m)e(kx),
\end{equation}
where
$$
b_0(y,s;m) = \frac{y^{1-s}}{(2s-1)\Gamma(s)}2\pi^s |m|^{s-1/2} L_s(m,0)=\frac{y^{1-s}}{(2s-1)}B_0(s;m)
$$
and $L_s(m,n)$ is defined by \eqref{eq:Selberg zeta intro}. Therefore
\begin{equation}\label{eq. rel B to L}
B_0(1;m)=2\pi\sqrt{|m|} L_1(m,0).
\end{equation}
Moreover, from the proof of \cite[Theorem 1]{GS83}, it follows that
\begin{equation}\label{eq. L bound}
L_1(m,0)=L_1(-m,0)=O_\Gamma(m),\quad \text{as  } m\to+\infty.
\end{equation}

For $k\neq 0$ we have that $b_k(y,s;m)=B_k(s;m)y^{1/2}K_{s-1/2}(2\pi |k|y),$
with
$$
B_k(s;m)= 2 \sum_{c>0}S(m,k;c)c^{-1}\cdot \left\{
                                            \begin{array}{ll}
                                              J_{2s-1}\left(\frac{4\pi}{c} \sqrt{mk}\right), & \textrm{\rm if \,}mk>0 \\
                                              I_{2s-1}\left(\frac{4\pi}{c} \sqrt{|mk|}\right), & \textrm{\rm if \,} mk<0.
                                            \end{array}
                                          \right..
$$
In the above expression, $J_{2s-1}$ denotes the $J$-Bessel function and $K_{s-1/2}$ is the modified Bessel function of the second kind. According to the proof of Theorem 6 from \cite{Ni73}, the Fourier expansion \eqref{Four exp Nieb} extends by the principle of analytic continuation to the case when $s=1$. Hence, by putting $B_k(1;m):= \lim_{s\downarrow 1} B_k(s;m)$, and using the special values of $I$-Bessel and $K$-Bessel functions of order $1/2$, we have that
\begin{equation}\label{Four exp Nieb at 1}
F_m(z,1)=\frac{\sinh(2\pi|m|y)}{\pi \sqrt{|m|}}e(mx) + B_0(1;m)+ \sum_{k\in\ZZ\setminus\{0\}}\frac{1}{2 \sqrt{|k|}} e^{-2\pi|k|y}B_k(1;m)e(kx).
\end{equation}

It is clear from \eqref{Four exp Nieb at 1} that for $n>0$ one has that
\begin{equation*}\label{F at s=1}
F_{-n}(z,1) = \frac{1}{2\pi\sqrt{n}}q_{z}^{-n} + B_0(1;-n)+ O(e^{-2\pi y})
\,\,\,\,\,\textrm{\rm as $z=x+iy \rightarrow \infty$}.
\end{equation*}

\subsection{Proof of Theorem~\ref{cor:relation H NP}}\label{sec proof D}

Let us recall that for the normalized cusp form $f(z)=\sum_{m\geq 1}c_m q_z^m$ we have for $m\geq 2$ that
$$
\Phi_m(\tau)= q_\tau^{-m}-c_mq_\tau^{-1}+sc_m+O(q_\tau).
$$
On a different front, applying \cite[Theorem 5]{Ni73} to $\Phi_m$ which is a $\Gamma$ invariant weakly holomorphic function with a pole at $\infty$, for $m\geq 2$ we deduce the following:
$$
\Phi_m(\tau)=2\pi\sqrt{m}F_{-m}(\tau,1)-2\pi c_mF_{-1}(\tau,1)+C(m),
$$
where $C(m)$ is chosen so that
$$
[q_\tau^0]\left(2\pi\sqrt{m}F_{-m}(\tau,1)-2\pi c_mF_{-1}(\tau,1)+C(m)\right)=sc_m.
$$
Therefore,
$$
C(m)=sc_m-2\pi\sqrt{m}B_0(1;-m)+2\pi c_mB_0(1;-1),
$$
meaning that for $m\geq 2$ we have that
\begin{equation}\label{eq. id for sum}
\Phi_m(\tau)=2\pi\sqrt{m}F_{-m}(\tau,1)-2\pi c_mF_{-1}(\tau,1)+sc_m-2\pi\sqrt{m}B_0(1;-m)+2\pi c_mB_0(1;-1)
\end{equation}

For a fixed $\tau$, we multiply the above identity with $q_z^m$, for $z$ such that $\Im(z)\gg \Im(\tau)$. From \cite[Section 1.1]{BJS25} it is known that the series
$$
\mathcal{H}(z,\tau):=2\pi\sum_{m\geq 1}\sqrt{m}F_{-m}(\tau,1)q_z^m
$$
converges absolutely and locally uniformly. Therefore, $z$ such that $\Im(z)\gg \Im(\tau)$, the bound \eqref{eq. L bound} ensures that we may sum the identity \eqref{eq. id for sum} over all $m\geq 2$ (all the series converge absolutely and locally uniformly). By using that $\sum_{m\geq 2}c_mq_z^m = f(z)-q_z$ we get
\begin{align*}
\sum_{m\geq 0}\Phi_m(\tau)q_z^m -1&=2\pi\sum_{m\geq 1}\sqrt{m}F_{-m}(\tau,1)q_z^m + f(z)(s+2\pi B_0(1;-1) - 2\pi F_{-1}(\tau,1)) \\ &- 2\pi \sum_{m\geq 1}\sqrt{m}B_0(1;-m)q_z^m.
\end{align*}

This proves \eqref{eq. identity H and NP}. Part (i) follows by combining Lemma~\ref{thmB} with the result of \cite{BJS25} that $\tau\mapsto\mathcal{H}(z,\tau)$ is a weight zero polar harmonic Maass form. Analogously, part (ii) follows by combining Lemma~\ref{lem:A} with the result of \cite[Thm~1.1(3)]{BJS25} that $z\mapsto\mathcal{H}(z,\tau) +({\rm vol}(\Gamma\backslash\HH) \Im(z))^{-1}$ is a weight two polar harmonic Maass form with the only pole in the fundamental domain at $z=\tau$ with the polar part equal to $\frac{i e_\tau}{2\pi (z-\tau)}.$
(This agrees with the polar structure of $\hgz(z,\tau)$ deduced in Lemma~\ref{lem:A}).

Finally, part (iii) follows from (ii), using the identity \eqref{eq. rel B to L}.

\section{Specialization to Atkin-Lehner groups of genus $1$}\label{sec: arith}

In this section we provide an arithmetic application of our main results and prove Theorem~\ref{thmI}  by specializing to Atkin-Lehner groups of genus one. We also present numerical computations.

Let us first recall the definition of Atkin-Lehner groups of a positive, square-free level $N=p_1\cdots p_r$, including the case $N=1$. A subgroup of $\text{\rm PSL}(2,\mathbb R)$ defined by
\begin{align}
\Gamma_0(N)^+: =\Big\{ \frac{1}{\sqrt{e}}\begin{pmatrix}a&b\\c&d\end{pmatrix}\in
  \text{\rm SL}(2,\mathbb R):  &  \, ad-bc=e, \, a,b,c,d,e\in\mathbb Z, \,e>0,  \\&  e\mid N,\ e\mid a,
   \ e\mid d,\ N\mid c \Big\} /\{\pm \textrm{Id}\}\nonumber
\end{align}
is called the Atkin-Lehner group of level $N$\footnote{This terminology is used because the group is obtained by adding all Atkin-Lehner involutions to the congruence group $\Gamma_0(N)$, see \cite{AtLeh70}.}.  In order to distinguish between different levels $N$, in the notation used above we will add the subscript $N$ to signify that the considered quantities and functions are associated to the group $\Gamma_0(N)^+$.

\subsection{Proof of Theorem~\ref{thmI}}

From Theorem 8 and Proposition 9 of \cite{JST}, in \cite{CJK23} it was deduced that for positive integers $m$
\begin{equation}\label{B0 for A-L groups}
B_{0,N}(1;m)=B_{0,N}(1;-m)= \frac{12\sigma(m)}{\pi\sqrt{m}}\prod\limits_{\nu=1}^r \left(1-
\frac{p_\nu^{\alpha_{p_\nu}(m)+1}(p_\nu -1) }{\left(p_\nu^{\alpha_{p_\nu}(m)+1} - 1\right)(p_\nu+1)}\right) ,
\end{equation}
where $\sigma(m)$ denotes the sum of divisors and $\alpha_p(m)$ is the largest integer such that $p^{\alpha_p(m)}$ divides $m$.
Moreover, ${\rm vol} (\Gamma_0(N)^+ \backslash \mathbb H)= \frac{\pi\sigma(N)}{3\cdot 2^r}$. 

Combining Theorem~\ref{cor:relation H NP} with \eqref{B0 for A-L groups} proves  Theorem~\ref{thmI}.

\subsection{Expansions of the basis $\{\Phi_0\}\cup\{\Phi_m\}_{m\geq 2}$}

We discuss arithmetic examples in the situation when $\Gamma$ is genus 1 Atkin-Lehner group, meaning that $N$ belongs to the set of $38$ integers between $37$ and $238$ listed in \cite{Cum04} for which the genus of the surface $\Gamma_0(N)^+ \backslash \mathbb H$ equals one.
The generators $x_N$ and $y_N$ of the function field of meromorphic functions on the surface $\Gamma_0(N)^+ \backslash \mathbb H$ have been deduced in \cite{JST} for all genus one levels and applied in \cite{JST20, JST21} to find generating polynomials of certain ring class fields. In the sequel we show how to use $x_N$ and $y_N$ in order to deduce the canonical basis of $M_{0,N}^{!,\infty}$. Starting with $x_N,\, y_N$,  one can use Theorem~\ref{thmC} to determine the basis $\{\Phi_0\}\cup\{\Phi_m\}_{m\geq 2}$  of the space $M_{0,N}^{!,\infty}$ of weakly holomorphic modular forms by equating coefficients in the $q$-expansion of the defining formula \eqref{eq:defnH} for $H(z,\tau)$.

In the setting of Atkin-Lehner groups of genus one the first computation of the canonical basis of weight $k$ weakly holomorphic modular forms was conducted in \cite{CK13} for prime levels (and used in \cite{CK22} to generate mock modular forms). Recently, in \cite{KL26}  a canonical basis of weight $k$ weakly holomorphic modular forms was computed for all non-negative weights and a wide class of arithmetic groups of genus zero and one, containing  Atkin-Lehner groups. The approach undertaken in \cite{KL26} is different from ours. Namely, instead of looking at the expansion of the generating kernel $H_N(z,\tau)$ the basis elements in \cite{KL26} are constructed as products of (variants of) a unique normalized cusp from of maximal vanishing order and polynomials in $x_N$. 

%\subsection{Computations}
\medskip
Assume that we have the $q-$expansions of $x(z),y(z),f(z)$ to arbitrary order. The computation of the $q$-expansion of the basis $\{\Phi_0\}\cup\{\Phi_m\}_{m\geq 2}$ is described in the following, simple algorithm.
\medskip

\noindent \textbf{Algorithm}
\begin{enumerate}[label=(\arabic*)]
\item Truncate $x(z),y(z),f(z)$ to order $q_z^{M}$.
\item Form $(y(z)+a_1x(z)+a_3+B) \cdot f(z)$ as a
$\CC[B][q_z^{\pm}]$-series; its lowest term is $q_z^{-2}$.
\item Write $x(z)-A=q_z^{-2}\bigl(1+s\,q_z-A\,q_z^{2}+x_1q_z^{3}+\cdots\bigr)$ and
invert the unit factor by the standard Newton/geometric recursion, obtaining
$1/(x(z)-A)=q_z^{2}\bigl(1-s\,q_z+(A+s^2)q_z^{2}+\cdots\bigr)\in\CC[A][[q_z]]\,q_z^2$.
\item Multiply: $\Phi_m(A,B)=[q_z^{\,m}]\bigl((y(z)+a_1x(z)+a_3+B) \cdot f(z)\cdot(x(z)-A)^{-1}\bigr)$.
\end{enumerate}

Initially, while doing a systematic study of level $N=37$ we used $q$-expansions of the generators $x_N$, $y_N$ for genus one levels $N$ listed in \cite{JSTdata} where the coefficients in the generalized Weierstrass model have been derived (see \cite[Table 6]{JST}). (The $q$-expansions of the  weight two cusp form $f_N(z)$ derived from \eqref{eq:omega} and \eqref{eq:fdef} have also  been checked against the data from \cite{LFMDB}).

However, the precision in \cite{JSTdata} was not high enough so we computed everything fresh \emph{except} the coefficients of the generalized Weierstrass model which we needed \cite[Table 6]{JST}. Our first step was computing $f(z)$ and we did so in two independent ways: one via the Modularity theorem relating $f$ to the Hasse–Weil $L-$function of the elliptic curve. The second way was using exact computations using Pari/GP \cite{PARI2} and SageMath \cite{sagemath} and its implementation of the Atkin-Lehner Hecke operators. Once we had the $q-$expansion of $f(z)$ we used a formal Newton iteration based on \eqref{eq:recursionQ} to very efficiently expand $x(z)$ and $y(z).$

We record $\Phi_m(A,B)=\Phi_m(x_N,y_N)$ and the $q_\tau$ expansions.  A computation is carried out for all genus one levels. In tables \ref{table 37} and \ref{table 238}, respectively, we illustrate results for the smallest genus one level $N=37$ and for the largest genus one level $N=238$.
We list the weight two cusp form, the generators $x$ and $y$ and expansion of generators $\Phi(m)$, for $m\in\{2,3,5,10,15,20\}$ and the corresponding polynomials $\Phi_m(A,B)$. Upon request, we can provide the complete data produced (roughly 100 pages).

\subsection{Computations for level $N = 37$}
\small
\begin{longtable}{>{\centering\arraybackslash}p{0.14\textwidth} | p{0.78\textwidth}}\label{table 37}
         \textbf{Object} & \textbf{Expansion} \\
    \hline
        \endhead
        $f(q)$ & $q - \allowbreak 2q^{2} - \allowbreak 3q^{3} + \allowbreak 2q^{4} - \allowbreak 2q^{5} + \allowbreak 6q^{6} - \allowbreak q^{7} + \allowbreak 6q^{9} + \allowbreak 4q^{10} - \allowbreak 5q^{11} - \allowbreak 6q^{12} - \allowbreak 2q^{13} + \allowbreak 2q^{14} + \allowbreak 6q^{15} - \allowbreak 4q^{16} - \allowbreak 12q^{18} - \allowbreak 4q^{20} + \allowbreak 3q^{21} + \allowbreak 10q^{22} + \allowbreak 2q^{23} - \allowbreak q^{25} + \allowbreak O(q^{26})$ \\
\hline
$x(q)$ & $q^{{-2}} + \allowbreak 2 q^{-1} + \allowbreak 9q + \allowbreak 18q^{2} + \allowbreak 29q^{3} + \allowbreak 51q^{4} + \allowbreak 82q^{5} + \allowbreak 131q^{6} + \allowbreak 199q^{7} + \allowbreak 306q^{8} + \allowbreak 450q^{9} + \allowbreak 666q^{10} + \allowbreak 957q^{11} + \allowbreak 1375q^{12} + \allowbreak 1934q^{13} + \allowbreak 2719q^{14} + \allowbreak 3752q^{15} + \allowbreak 5174q^{16} + \allowbreak 7040q^{17} + \allowbreak 9546q^{18} + \allowbreak 12812q^{19} + \allowbreak 17146q^{20} + \allowbreak 22735q^{21} + \allowbreak 30062q^{22} + \allowbreak 39450q^{23} + \allowbreak 51606q^{24} + \allowbreak 67087q^{25} + \allowbreak O(q^{26})$ \\
\hline
$y(q)$ & $q^{{-3}} + \allowbreak 3 q^{-1} + \allowbreak 19q + \allowbreak 38q^{2} + \allowbreak 93q^{3} + \allowbreak 176q^{4} + \allowbreak 347q^{5} + \allowbreak 630q^{6} + \allowbreak 1139q^{7} + \allowbreak 1944q^{8} + \allowbreak 3305q^{9} + \allowbreak 5404q^{10} + \allowbreak 8762q^{11} + \allowbreak 13848q^{12} + \allowbreak 21667q^{13} + \allowbreak 33262q^{14} + \allowbreak 50609q^{15} + \allowbreak 75836q^{16} + \allowbreak 112683q^{17} + \allowbreak 165464q^{18} + \allowbreak 241038q^{19} + \allowbreak 347680q^{20} + \allowbreak 497979q^{21} + \allowbreak 707254q^{22} + \allowbreak 998056q^{23} + \allowbreak 1398288q^{24} + \allowbreak 1947542q^{25} + \allowbreak O(q^{26})$ \\
\hline
$\Phi_{2}(A, B)$ & $A - \allowbreak 4$ \\
\hline
$\Phi_{2}(q)$ & $q^{{-2}} + \allowbreak 2 q^{-1} - \allowbreak 4 + \allowbreak 9q + \allowbreak 18q^{2} + \allowbreak 29q^{3} + \allowbreak 51q^{4} + \allowbreak 82q^{5} + \allowbreak 131q^{6} + \allowbreak 199q^{7} + \allowbreak 306q^{8} + \allowbreak 450q^{9} + \allowbreak 666q^{10} + \allowbreak 957q^{11} + \allowbreak 1375q^{12} + \allowbreak 1934q^{13} + \allowbreak 2719q^{14} + \allowbreak 3752q^{15} + \allowbreak 5174q^{16} + \allowbreak 7040q^{17} + \allowbreak 9546q^{18} + \allowbreak 12812q^{19} + \allowbreak 17146q^{20} + \allowbreak 22735q^{21} + \allowbreak 30062q^{22} + \allowbreak 39450q^{23} + \allowbreak 51606q^{24} + \allowbreak 67087q^{25} + \allowbreak O(q^{26})$ \\
\hline
$\Phi_{3}(A, B)$ & $B - \allowbreak 6$ \\
\hline
$\Phi_{3}(q)$ & $q^{{-3}} + \allowbreak 3 q^{-1} - \allowbreak 6 + \allowbreak 19q + \allowbreak 38q^{2} + \allowbreak 93q^{3} + \allowbreak 176q^{4} + \allowbreak 347q^{5} + \allowbreak 630q^{6} + \allowbreak 1139q^{7} + \allowbreak 1944q^{8} + \allowbreak 3305q^{9} + \allowbreak 5404q^{10} + \allowbreak 8762q^{11} + \allowbreak 13848q^{12} + \allowbreak 21667q^{13} + \allowbreak 33262q^{14} + \allowbreak 50609q^{15} + \allowbreak 75836q^{16} + \allowbreak 112683q^{17} + \allowbreak 165464q^{18} + \allowbreak 241038q^{19} + \allowbreak 347680q^{20} + \allowbreak 497979q^{21} + \allowbreak 707254q^{22} + \allowbreak 998056q^{23} + \allowbreak 1398288q^{24} + \allowbreak 1947542q^{25} + \allowbreak O(q^{26})$ \\
\hline
$\Phi_{5}(A, B)$ & $-2 A^{2} + \allowbreak A B - \allowbreak 7 A + \allowbreak 5 B + \allowbreak 8$ \\
\hline
$\Phi_{5}(q)$ & $q^{{-5}} + \allowbreak 2 q^{-1} - \allowbreak 4 + \allowbreak 46q + \allowbreak 168q^{2} + \allowbreak 545q^{3} + \allowbreak 1492q^{4} + \allowbreak 3820q^{5} + \allowbreak 8864q^{6} + \allowbreak 19663q^{7} + \allowbreak 41340q^{8} + \allowbreak 83850q^{9} + \allowbreak 163984q^{10} + \allowbreak 312090q^{11} + \allowbreak 578092q^{12} + \allowbreak 1047814q^{13} + \allowbreak 1859716q^{14} + \allowbreak 3242321q^{15} + \allowbreak 5557620q^{16} + \allowbreak 9385244q^{17} + \allowbreak 15626100q^{18} + \allowbreak 25688555q^{19} + \allowbreak 41725576q^{20} + \allowbreak 67033468q^{21} + \allowbreak 106576372q^{22} + \allowbreak 167823384q^{23} + \allowbreak 261866676q^{24} + \allowbreak 405145708q^{25} + \allowbreak O(q^{26})$ \\
\hline
$\Phi_{10}(A, B)$ & $A^{5} + \allowbreak 20 A^{4} - \allowbreak 10 A^{3} B + \allowbreak 60 A^{3} - \allowbreak 135 A^{2} B - \allowbreak 731 A^{2} - \allowbreak 582 A B - \allowbreak 4831 A - \allowbreak 805 B - \allowbreak 7876$ \\
\hline
$\Phi_{10}(q)$ & $q^{{-10}} + \allowbreak -4 q^{-1} + \allowbreak 8 + \allowbreak 318q + \allowbreak 2994q^{2} + \allowbreak 17670q^{3} + \allowbreak 82746q^{4} + \allowbreak 327804q^{5} + \allowbreak 1156467q^{6} + \allowbreak 3719034q^{7} + \allowbreak 11116120q^{8} + \allowbreak 31251300q^{9} + \allowbreak 83453640q^{10} + \allowbreak 213150830q^{11} + \allowbreak 523739466q^{12} + \allowbreak 1243533132q^{13} + \allowbreak 2863838113q^{14} + \allowbreak 6416674872q^{15} + \allowbreak 14024175960q^{16} + \allowbreak 29964085072q^{17} + \allowbreak 62705908590q^{18} + \allowbreak 128739024600q^{19} + \allowbreak 259676030284q^{20} + \allowbreak 515253334554q^{21} + \allowbreak 1006848274446q^{22} + \allowbreak 1939517635652q^{23} + \allowbreak 3686361979608q^{24} + \allowbreak 6918709481946q^{25} + \allowbreak O(q^{26})$ \\
\hline
$\Phi_{15}(A, B)$ & $-12 A^{7} + \allowbreak A^{6} B - \allowbreak 292 A^{6} + \allowbreak 105 A^{5} B - \allowbreak 2109 A^{5} + \allowbreak 2018 A^{4} B + \allowbreak 2378 A^{4} + \allowbreak 16533 A^{3} B + \allowbreak 107317 A^{3} + \allowbreak 67626 A^{2} B + \allowbreak 569700 A^{2} + \allowbreak 135844 A B + \allowbreak 1265020 A + \allowbreak 106386 B + \allowbreak 1044936$ \\
\hline
$\Phi_{15}(q)$ & $q^{{-15}} + \allowbreak -6 q^{-1} + \allowbreak 12 + \allowbreak 1597q + \allowbreak 26516q^{2} + \allowbreak 251410q^{3} + \allowbreak 1733924q^{4} + \allowbreak 9726269q^{5} + \allowbreak 46877208q^{6} + \allowbreak 201098126q^{7} + \allowbreak 785596140q^{8} + \allowbreak 2840958640q^{9} + \allowbreak 9625012756q^{10} + \allowbreak 30832356830q^{11} + \allowbreak 94058739544q^{12} + \allowbreak 274844284558q^{13} + \allowbreak 772880003512q^{14} + \allowbreak 2099769571800q^{15} + \allowbreak 5529542105420q^{16} + \allowbreak 14153887874883q^{17} + \allowbreak 35299667276920q^{18} + \allowbreak 85955628644820q^{19} + \allowbreak 204725984022388q^{20} + \allowbreak 477702935321056q^{21} + \allowbreak 1093555621680724q^{22} + \allowbreak 2459044582812568q^{23} + \allowbreak 5437780437156432q^{24} + \allowbreak 11837083082322323q^{25} + \allowbreak O(q^{26})$ \\
\hline
$\Phi_{20}(A, B)$ & $A^{10} + \allowbreak 140 A^{9} - \allowbreak 20 A^{8} B + \allowbreak 3820 A^{8} - \allowbreak 1270 A^{7} B + \allowbreak 42103 A^{7} - \allowbreak 28064 A^{6} B + \allowbreak 154828 A^{6} - \allowbreak 316360 A^{5} B - \allowbreak 974663 A^{5} - \allowbreak 2071947 A^{4} B - \allowbreak 13538407 A^{4} - \allowbreak 8223162 A^{3} B - \allowbreak 66877823 A^{3} - \allowbreak 19496464 A^{2} B - \allowbreak 174301555 A^{2} - \allowbreak 25371536 A B - \allowbreak 237718245 A - \allowbreak 13905709 B - \allowbreak 133504196$ \\
\hline
$\Phi_{20}(q)$ & $q^{{-20}} + \allowbreak 4 q^{-1} - \allowbreak 8 + \allowbreak 5932q + \allowbreak 165546q^{2} + \allowbreak 2311960q^{3} + \allowbreak 22232454q^{4} + \allowbreak 166899968q^{5} + \allowbreak 1047479398q^{6} + \allowbreak 5727637696q^{7} + \allowbreak 28048353210q^{8} + \allowbreak 125411651280q^{9} + \allowbreak 519352063068q^{10} + \allowbreak 2013695928540q^{11} + \allowbreak 7372723964999q^{12} + \allowbreak 25663761554648q^{13} + \allowbreak 85404828841922q^{14} + \allowbreak 272967978699316q^{15} + \allowbreak 841172839882160q^{16} + \allowbreak 2507421877813168q^{17} + \allowbreak 7250373947410740q^{18} + \allowbreak 20386482808783720q^{19} + \allowbreak 55859425561064720q^{20} + \allowbreak 149428836363470276q^{21} + \allowbreak 390907494780269074q^{22} + \allowbreak 1001506736441573568q^{23} + \allowbreak 2516200660694157642q^{24} + \allowbreak 6206711535594291204q^{25} + \allowbreak O(q^{26})$ \\
\hline
\end{longtable}

\subsection{Computations for level $N = 238$}

\begin{longtable}{>{\centering\arraybackslash}p{0.14\textwidth} | p{0.78\textwidth}}\label{table 238}
        \textbf{Object} & \textbf{Expansion} \\
        \hline
        \endhead
        $f(q)$ & $q - \allowbreak q^{2} + \allowbreak q^{4} - \allowbreak 2q^{5} - \allowbreak q^{7} - \allowbreak q^{8} - \allowbreak 3q^{9} + \allowbreak 2q^{10} - \allowbreak 2q^{11} + \allowbreak q^{14} + \allowbreak q^{16} - \allowbreak q^{17} + \allowbreak 3q^{18} - \allowbreak 2q^{19} - \allowbreak 2q^{20} + \allowbreak 2q^{22} - \allowbreak 8q^{23} - \allowbreak q^{25} + \allowbreak O(q^{26})$ \\
\hline
$x(q)$ & $q^{{-2}} + \allowbreak q^{-1} + \allowbreak q^{3} + \allowbreak q^{4} + \allowbreak q^{5} + \allowbreak 2q^{6} + \allowbreak q^{7} + \allowbreak 2q^{8} + \allowbreak q^{9} + \allowbreak 2q^{10} + \allowbreak 2q^{11} + \allowbreak 2q^{12} + \allowbreak 2q^{13} + \allowbreak 3q^{14} + \allowbreak 2q^{15} + \allowbreak 3q^{16} + \allowbreak 3q^{17} + \allowbreak 4q^{18} + \allowbreak 3q^{19} + \allowbreak 5q^{20} + \allowbreak 4q^{21} + \allowbreak 6q^{22} + \allowbreak 5q^{23} + \allowbreak 6q^{24} + \allowbreak 6q^{25} + \allowbreak O(q^{26})$ \\
\hline
$y(q)$ & $q^{{-3}} + \allowbreak q + \allowbreak q^{2} + \allowbreak q^{3} + \allowbreak q^{4} + \allowbreak 2q^{5} + \allowbreak q^{6} + \allowbreak 2q^{7} + \allowbreak q^{8} + \allowbreak 4q^{9} + \allowbreak 2q^{10} + \allowbreak 5q^{11} + \allowbreak 4q^{12} + \allowbreak 6q^{13} + \allowbreak 5q^{14} + \allowbreak 8q^{15} + \allowbreak 7q^{16} + \allowbreak 10q^{17} + \allowbreak 9q^{18} + \allowbreak 13q^{19} + \allowbreak 11q^{20} + \allowbreak 17q^{21} + \allowbreak 14q^{22} + \allowbreak 20q^{23} + \allowbreak 18q^{24} + \allowbreak 26q^{25} + \allowbreak O(q^{26})$ \\
\hline
$\Phi_{2}(A, B)$ & $A - \allowbreak 1$ \\
\hline
$\Phi_{2}(q)$ & $q^{{-2}} + \allowbreak q^{-1} - \allowbreak 1 + \allowbreak q^{3} + \allowbreak q^{4} + \allowbreak q^{5} + \allowbreak 2q^{6} + \allowbreak q^{7} + \allowbreak 2q^{8} + \allowbreak q^{9} + \allowbreak 2q^{10} + \allowbreak 2q^{11} + \allowbreak 2q^{12} + \allowbreak 2q^{13} + \allowbreak 3q^{14} + \allowbreak 2q^{15} + \allowbreak 3q^{16} + \allowbreak 3q^{17} + \allowbreak 4q^{18} + \allowbreak 3q^{19} + \allowbreak 5q^{20} + \allowbreak 4q^{21} + \allowbreak 6q^{22} + \allowbreak 5q^{23} + \allowbreak 6q^{24} + \allowbreak 6q^{25} + \allowbreak O(q^{26})$ \\
\hline
$\Phi_{3}(A, B)$ & $B$ \\
\hline
$\Phi_{3}(q)$ & $q^{{-3}} + \allowbreak q + \allowbreak q^{2} + \allowbreak q^{3} + \allowbreak q^{4} + \allowbreak 2q^{5} + \allowbreak q^{6} + \allowbreak 2q^{7} + \allowbreak q^{8} + \allowbreak 4q^{9} + \allowbreak 2q^{10} + \allowbreak 5q^{11} + \allowbreak 4q^{12} + \allowbreak 6q^{13} + \allowbreak 5q^{14} + \allowbreak 8q^{15} + \allowbreak 7q^{16} + \allowbreak 10q^{17} + \allowbreak 9q^{18} + \allowbreak 13q^{19} + \allowbreak 11q^{20} + \allowbreak 17q^{21} + \allowbreak 14q^{22} + \allowbreak 20q^{23} + \allowbreak 18q^{24} + \allowbreak 26q^{25} + \allowbreak O(q^{26})$ \\
\hline
$\Phi_{5}(A, B)$ & $-A^{2} + \allowbreak A B + \allowbreak A + \allowbreak 2 B - \allowbreak 5$ \\
\hline
$\Phi_{5}(q)$ & $q^{{-5}} + \allowbreak 2 q^{-1} - \allowbreak 2 + \allowbreak 3q + \allowbreak q^{2} + \allowbreak 4q^{3} + \allowbreak 2q^{4} + \allowbreak 6q^{5} + \allowbreak 4q^{6} + \allowbreak 9q^{7} + \allowbreak 10q^{8} + \allowbreak 13q^{9} + \allowbreak 13q^{10} + \allowbreak 21q^{11} + \allowbreak 19q^{12} + \allowbreak 29q^{13} + \allowbreak 29q^{14} + \allowbreak 45q^{15} + \allowbreak 41q^{16} + \allowbreak 62q^{17} + \allowbreak 62q^{18} + \allowbreak 87q^{19} + \allowbreak 87q^{20} + \allowbreak 118q^{21} + \allowbreak 124q^{22} + \allowbreak 165q^{23} + \allowbreak 169q^{24} + \allowbreak 222q^{25} + \allowbreak O(q^{26})$ \\
\hline
$\Phi_{10}(A, B)$ & $A^{5} + \allowbreak 5 A^{4} - \allowbreak 5 A^{3} B - \allowbreak 15 A^{2} B + \allowbreak 11 A^{2} - \allowbreak 6 A B + \allowbreak 4 A - \allowbreak 2 B - \allowbreak 5$ \\
\hline
$\Phi_{10}(q)$ & $q^{{-10}} + \allowbreak -2 q^{-1} + \allowbreak 2 + \allowbreak 2q + \allowbreak 9q^{2} + \allowbreak 6q^{3} + \allowbreak 23q^{4} + \allowbreak 24q^{5} + \allowbreak 46q^{6} + \allowbreak 56q^{7} + \allowbreak 100q^{8} + \allowbreak 122q^{9} + \allowbreak 202q^{10} + \allowbreak 244q^{11} + \allowbreak 376q^{12} + \allowbreak 456q^{13} + \allowbreak 681q^{14} + \allowbreak 826q^{15} + \allowbreak 1194q^{16} + \allowbreak 1448q^{17} + \allowbreak 2043q^{18} + \allowbreak 2478q^{19} + \allowbreak 3379q^{20} + \allowbreak 4112q^{21} + \allowbreak 5521q^{22} + \allowbreak 6690q^{23} + \allowbreak 8821q^{24} + \allowbreak 10694q^{25} + \allowbreak O(q^{26})$ \\
\hline
$\Phi_{15}(A, B)$ & $-6 A^{7} + \allowbreak A^{6} B - \allowbreak 29 A^{6} + \allowbreak 27 A^{5} B - \allowbreak 36 A^{5} + \allowbreak 98 A^{4} B - \allowbreak 80 A^{4} + \allowbreak 119 A^{3} B - \allowbreak 96 A^{3} + \allowbreak 78 A^{2} B - \allowbreak 63 A^{2} - \allowbreak 8 A B + \allowbreak 5 A - \allowbreak 14 B + \allowbreak 20$ \\
\hline
$\Phi_{15}(q)$ & $q^{{-15}} + \allowbreak 10q + \allowbreak 10q^{2} + \allowbreak 40q^{3} + \allowbreak 55q^{4} + \allowbreak 131q^{5} + \allowbreak 185q^{6} + \allowbreak 350q^{7} + \allowbreak 505q^{8} + \allowbreak 890q^{9} + \allowbreak 1241q^{10} + \allowbreak 2045q^{11} + \allowbreak 2865q^{12} + \allowbreak 4485q^{13} + \allowbreak 6170q^{14} + \allowbreak 9326q^{15} + \allowbreak 12700q^{16} + \allowbreak 18640q^{17} + \allowbreak 25165q^{18} + \allowbreak 36025q^{19} + \allowbreak 48206q^{20} + \allowbreak 67580q^{21} + \allowbreak 89690q^{22} + \allowbreak 123560q^{23} + \allowbreak 162715q^{24} + \allowbreak 220922q^{25} + \allowbreak O(q^{26})$ \\
\hline
$\Phi_{20}(A, B)$ & $A^{10} + \allowbreak 35 A^{9} - \allowbreak 10 A^{8} B + \allowbreak 175 A^{8} - \allowbreak 155 A^{7} B + \allowbreak 357 A^{7} - \allowbreak 637 A^{6} B + \allowbreak 668 A^{6} - \allowbreak 1179 A^{5} B + \allowbreak 976 A^{5} - \allowbreak 1271 A^{4} B + \allowbreak 1015 A^{4} - \allowbreak 618 A^{3} B + \allowbreak 602 A^{3} + \allowbreak 54 A^{2} B + \allowbreak 182 A B - \allowbreak 174 A + \allowbreak 50 B - \allowbreak 45$ \\
\hline
$\Phi_{20}(q)$ & $q^{{-20}} + \allowbreak 2 q^{-1} - \allowbreak 2 + \allowbreak 8q + \allowbreak 46q^{2} + \allowbreak 74q^{3} + \allowbreak 207q^{4} + \allowbreak 346q^{5} + \allowbreak 754q^{6} + \allowbreak 1234q^{7} + \allowbreak 2395q^{8} + \allowbreak 3818q^{9} + \allowbreak 6760q^{10} + \allowbreak 10516q^{11} + \allowbreak 17654q^{12} + \allowbreak 26844q^{13} + \allowbreak 43114q^{14} + \allowbreak 64276q^{15} + \allowbreak 99846q^{16} + \allowbreak 146222q^{17} + \allowbreak 220882q^{18} + \allowbreak 318742q^{19} + \allowbreak 470700q^{20} + \allowbreak 669888q^{21} + \allowbreak 970184q^{22} + \allowbreak 1363970q^{23} + \allowbreak 1942724q^{24} + \allowbreak 2700508q^{25} + \allowbreak O(q^{26})$ \\
\hline
\end{longtable}

\vspace{5mm}

\noindent
Joshua S. Friedman \\
Department of Mathematics and Science \\
\textsc{United States Merchant Marine Academy} \\
300 Steamboat Road \\
Kings Point, NY 11024 \\
U.S.A. \\
e-mail: FriedmanJ@usmma.edu, joshua@math.sunysb.edu, CrownEagle@gmail.com

\vspace{5mm}
\noindent
Jay Jorgenson \\
Department of Mathematics \\
The City College of New York \\
Convent Avenue at 138th Street \\
New York, NY 10031
U.S.A. \\
e-mail: jjorgenson@mindspring.com

\vspace{5mm}

\noindent
Lejla Smajlovi\'c \\
Department of Mathematics and Computer Science \\
University of Sarajevo\\
Zmaja od Bosne 35, 71 000 Sarajevo\\
Bosnia and Herzegovina\\
e-mail: lejlas@pmf.unsa.ba

\normalsize

\begin{thebibliography}{9}

\bibitem{Al03} S. Ahlgren, \emph{The theta-operator and the divisors of modular forms on genus zero
subgroups},. Math. Res. Lett. 10 (2003), 787--798.

\bibitem{AKN} T.~Asai, M.~Kaneko, H.~Ninomiya,
\emph{Zeros of certain modular functions and an application},
Comment. Math. Univ. St. Paul. \textbf{46} (1997), 93--101.

\bibitem{AtLeh70} A.~O.~L. Atkin and J. Lehner,  \emph{Hecke operators on $\Gamma_0(m)$},
Math.\ Ann.\ \textbf{185} (1970), 134--160.

\bibitem{BS25} A. Bhand, R. K. Singh,
\emph{Generalisation of the Asai-Kaneko-Ninomiya identity to higher level}, J. Math. Anal. Appl. 542, No. 2, Article ID 128849, 24 p. (2025).

\bibitem{BM21} L. Beneish, M. H. Mertens, \emph{On Weierstrass mock modular forms and a dimension formula for certain vertex operator algebras}, Math. Z. 297, No. 1-2 (2021), 59--80.

\bibitem{BL15} L. Beneish, H. Larson, \emph{Traces of singular values of Hauptmoduln}, Int. J. Number Theory, 11 (2015), 1027--1048.

\bibitem{BJS25} K. Bringmann, J. Jorgenson, L. Smajlović \emph{On a generating function of Niebur-Poincaré series}, Preprint, arXiv:2512.13167 [math.NT] (2025).

\bibitem{BK16} K. Bringmann, B. Kane, A problem of Petersson about weight 0 meromorphic modular forms, Res. Math.
Sci. 3 (2016), paper no. 24.

\bibitem{BKLOR18} K. Bringmann, B. Kane, S. L\"obrich, K. Ono, L. Rolen, On divisors of modular forms, Adv. Math. 329 (2018), 541--554.

\bibitem{BKO04} J. Bruinier, W. Kohnen, and K. Ono, The arithmetic of the values of modular functions and the divisors of
modular forms, Compos. Math. 140 (2004), no. 3, 552--566.

\bibitem{Ca10} S. Carnahan, \emph{Generalized moonshine I: Genus-zero functions}, Algebra and Number
Theory 4 No.6 (2010), 649--679.

\bibitem{Ca12} S. Carnahan, \emph{Generalized moonshine, II: Borcherds products}, Duke Math. J. 161 No.5 (2012), 893--950.

\bibitem{CK13} S. Choi, C. H. Kim,  \emph{Basis for the space of weakly holomorphic modular forms in higher level cases}, J. Number Theory 133(4) (2013), 1300--1311.

\bibitem{CK22} S. Choi, C. H. Kim,  \emph{Explicit construction of mock modular forms from weakly holomorphic Hecke eigenforms}, Open Math. 20(1) (2022), 313--332.

\bibitem{CJK23} J. W. Cogdell, J. Jorgenson, L. Smajlović, \emph{Kronecker limit functions and an extension of the Rohrlich-Jensen formula}, Nagoya Math. J. 252 (2023), 810--841.

\bibitem{Cum04}
C.~J.\ Cummins,
\emph{Congruence subgroups of groups commensurable with ${\mathrm PSL}(2,\mathbb Z)$ of genus $0$ and $1$},
Experiment.\ Math.\ \textbf{13} (2004), 361--382.

%\bibitem{DS} F.~Diamond, J.~Shurman,
%\emph{A First Course in Modular Forms}, GTM 228, Springer, 2005.

\bibitem{DJ08} W. Duke, P. Jenkins, \emph{On the zeros and coefficients of certain weakly holomorphic modular forms}, Pure Appl. Math. Q. 4 (2008), no. 4, Special Issue: In honor of Jean-Pierre Serre. Part 1, 1327--1340.

    \bibitem{E-G09}  A. El--Guindy, \emph{Fourier expansions with modular form coefficients}, Int. J. Number Theory 5 (2009), no. 8, 1433--1446.


\bibitem{Fa1903} G. Faber, \"Uber polynomische Entwicklungen, Math. Ann. 57 (1903), 389--408.

\bibitem{FK} H.~M.~Farkas, I.~Kra,
\emph{Riemann Surfaces}, 2nd ed., GTM 71, Springer, 1992.

\bibitem{GS83}  D. Goldfeld, P. Sarnak, \emph{Sums of Kloosterman sums}, Invent. Math. 71  no. 2 (1983), 243--250.


\bibitem{HJ14} A. Haddock, P. Jenkins, \emph{Zeros of weakly holomorphic modular forms of level 4}, Int. J. Number Theory 10 (2014), no. 2, 455--470.

\bibitem{PARI2}
    The PARI~Group, PARI/GP version \texttt{2.15.4}, Univ. Bordeaux, 2023,
    \url{http://pari.math.u-bordeaux.fr/}.


\bibitem{sagemath}
 SageMath, the Sage Mathematics Software System
\url{https://www.sagemath.org}



%\bibitem{Hartshorne} R.~Hartshorne, \emph{Algebraic Geometry}, GTM 52, Springer, 1977.

%\bibitem{Iwa02} H. Iwaniec,  \emph{Spectral methods of automorphic forms}.
%Graduate Studies in Mathematics \textbf{53}, American Mathematical Society, Providence, RI, 2002.

\bibitem{JM19} P. Jenkins, G. Molnar, \emph{Zagier duality for level p
 weakly holomorphic modular forms}, Ramanujan J. 50, No. 1 (2019), 93--109.

\bibitem{JKK23} D. Jeon, S.--Y. Kang, C. H. Kim, \emph{The Hecke system of harmonic Maass functions and applications to modular curves of higher genera}, The Ramanujan Journal 62 (2023), 675--717.

\bibitem{JST} J.~Jorgenson, L.~Smajlovi\'c, H.~Then,
\emph{Kronecker's limit formula, holomorphic modular functions, and
$q$-expansions on certain arithmetic groups},
Exp. Math. \textbf{25} (2016), no.~3, 295--319.

%\bibitem{JST16b} J. Jorgenson, L. Smajlovi\'c, H. Then,  \emph{Certain aspects of holomorphic function theory on some genus-zero arithmetic groups}, LMS J. Comput. Math. 19(2) (2016), 360--381.

\bibitem{JSTdata} J. Jorgenson, L. Smajlovi\'c, H. Then, Data page, Available at \url{http://www.efsa.unsa.ba/\~lejla.smajlovic/jst2/}.

\bibitem{JST20} J.~Jorgenson, L.~Smajlovi\'c, H.~Then,
\emph{On the evaluation of singular invariants for canonical generators of certain genus one arithmetic groups}, Exp. Math. 29 no. 1 (2020), 1--27.

\bibitem{JST21} J.~Jorgenson, L.~Smajlovi\'c, H.~Then, \emph{An approach for computing generators of class fields of imaginary quadratic number fields using the Schwarzian derivative}, Math. Comp. 91 no. 333 (2021), 331--379.


\bibitem{KL26} C. H. Kim, K. S. Lee (02 Jun 2026): Basis of Weakly
Holomorphic Modular Form Spaces for Squarefree Level Cases, Experimental Mathematics (2026), published online, DOI: 10.1080/10586458.2026.2622042.

\bibitem{LFMDB} The LMFDB Collaboration, The L-functions and modular forms database. Available at \url{http://www.lmfdb.org}.

\bibitem{Miranda} R.~Miranda,
\emph{Algebraic Curves and Riemann Surfaces}, GSM 5, Amer. Math. Soc., 1995.

\bibitem{Ni73} D. Niebur, \emph{A class of nonanalytic automorphic functions}, Nagoya Math. J. \textbf{52} (1973), 133--145.

\bibitem{Se65} A. Selberg, \emph{On the estimation of Fourier coefficients of modular forms}, in: Theory of Numbers, Proc. Sympos. Pure Math. 8, AMS (1965), 1--15.



\bibitem{Sil} J.~H.~Silverman,
\emph{The Arithmetic of Elliptic Curves}, 2nd ed., GTM 106, Springer, 2009.

\bibitem{Ye19} D. Ye, \emph{On the generating function of a canonical basis for} $M^{!,\infty}_0(\Gamma)$, Results Math., 74 (2019), 1--11.

\bibitem{Ye22} D. Ye, \emph{Difference of a Hauptmodul for} $\Gamma_0(N)$, Sci. China Math. 65 (2022), 221--258.

\bibitem{Za02} D. Zagier, Traces of singular moduli, Motives, polylogarithms and Hodge theory, Part I (Irvine, CA, 1998),
211--244, Int. Press Lect. Ser., 3, I, Int. Press, Somerville, MA, 2002.


%\bibitem{X37note}
%\emph{A bivariate Cauchy-kernel generating function for the function-field
%basis on $X_{37}=\Gamma_0(37)^+\backslash\HH^\ast$}, manuscript, 2026.

\end{thebibliography}
\end{document}